\documentclass[a4paper]{amsart}

\usepackage[all,hyperref,numberbysection]{bi-discrete}
\usepackage[backend=biber,maxbibnames=5,maxalphanames=5,style=alphabetic,bibencoding=utf8,giveninits,url=false,isbn=false]{biblatex}
\usepackage{xurl}
\hypersetup{breaklinks=true}
\AtBeginBibliography{\small}

\usepackage{xcolor}
\usepackage{microtype}
\usepackage{amssymb}
\numberwithin{equation}{section}

\allowdisplaybreaks

\author{Lars Diening}
\address{Fakult\"at f\"ur Mathematik, Universit\"at Bielefeld, Postfach 100131, D-33501 Bielefeld, Germany}
\email{lars.diening@uni-bielefeld.de}

\author{Simon Nowak}
\address{Fakult\"at f\"ur Mathematik, Universit\"at Bielefeld, Postfach 100131, D-33501 Bielefeld, Germany}
\email{simon.nowak@uni-bielefeld.de}

 \DeclareMathOperator*{\osc}{osc}


\begin{document}
	
	\title{Calder\'on-Zygmund estimates for the fractional $p$-Laplacian}
	
	\begin{abstract}
		We prove fine higher regularity results of Calder\'on-Zygmund-type for equations involving nonlocal operators modelled on the fractional $p$-Laplacian with possibly discontinuous coefficients of VMO-type.
		We accomplish this by establishing precise pointwise bounds in terms of certain fractional sharp maximal functions. This approach is new already in the linear setting and enables us to deduce sharp regularity results also in borderline cases.
	\end{abstract}
	
	\thanks{Funded by the Deutsche Forschungsgemeinschaft (DFG, German Research Foundation) - SFB 1283/2 2021 - 317210226}
	
	\maketitle

	\section{Introduction}
	\subsection{Overview} \label{sec:overview}
	In this paper, we study nonlocal integro-differential operators of \emph{fractional $p$-Laplacian-type} formally defined by
	\begin{equation} \label{eq:splap}
	(-\Delta )^s_{p,A} u(x) := p.v. \int_{\mathbb{R}^n} \frac{A(x,y)}{|x-y|^{n+sp}} |u(x)-u(y)|^{p-2}(u(x)-u(y)) dy,
	\end{equation}
	where $s \in (0,1)$, $p \in [2,\infty)$ and $A$ is some coefficient. 
	Although we consider such general nonlinear operators, our approach and some of our main results are new already in the linear case when $p=2$. Moreover, for $A \equiv 1$ the operator $(-\Delta )^s_{p,A}$ reduces to the standard fractional $p$-Laplacian $(-\Delta )^s_p$, in which case all of our main results already represent nontrivial and new developments.
	
	Nonlocal operators of the type \eqref{eq:splap} have attracted a lot of attention in recent years (see e.g.\ \cite{DKP,KMS2,SchikorraMA,CozziJFA,BL,BLS,CaffSire,KLL}) and arise for instance naturally in the calculus of variations, since they appear in the Euler-Lagrange equation of the weighted $W^{s,p}$ energy
	$$ u \mapsto \int_{\mathbb{R}^n}\int_{\mathbb{R}^n} A(x,y) \frac{|u(x)-u(y)|^p}{|x-y|^{n+sp}}dydx.$$
	Further areas in which nonlocal operators arise are for instance stochastic processes (see e.g.\ \cite{bertoin,Fukushima}), harmonic maps (see e.g.\ \cite{SchikorraHM}), classical harmonic analysis (see e.g.\ \cite{landkof}), phase transitions (see e.g.\ \cite{fife}), conformal geometry (see e.g.\ \cite{CG,Case-Chang,GZ}), physics of materials and relativistic models (see e.g.\ \cite{LiebYau2}), fluid dynamics (see e.g.\ \cite{NInvent,CaffVass}), kinetic theory (see e.g.\ \cite{ImbSil}) and image processing (see e.g.\ \cite{GOsh}).
	
	The purpose of this work is to understand the optimal regularity of weak solutions to nonlocal equations of the form
	\begin{equation} \label{eq:splapeq}
		(-\Delta )^s_{p,A} u = f \text{ in } \Omega \subset \mathbb{R}^n
	\end{equation}
	in terms of the data $f$ for a large class of coefficients $A$, including discontinuous ones. 
	
	In particular, our main results extend, sharpen and unify various previous regularity results obtained either for the standard fractional $p$-Laplacian or for linear nonlocal equations with possibly discontinuous coefficients of VMO-type.
	
	In fact, in the case of the standard fractional $p$-Laplacian with $p \geq 2$, in the papers \cite{BL,BLS} Brasco, Lindgren and Schikorra in particular proved the following fundamental higher regularity results:
	
	\begin{itemize}
		\item Higher differentiability: If $s \leq \frac{p-1}{p}$, then
		\begin{equation} \label{eq:HDBL}
		f \in W^{s,\frac{p}{p-1}}(\Omega) \implies u \in W^{t,p}_{\loc}(\Omega) \text{ for all } t \in \left (s, \frac{sp}{p-1} \right ).
		\end{equation}
		\item Higher H\"older regularity: For all $\alpha \in \left (0, \min \left \{\frac{sp}{p-1},1 \right \} \right )$, 
		\begin{equation} \label{eq:HHBLS}
		f \in L^{\frac{n}{sp-\alpha(p-1)}+\varepsilon}(\Omega) \text{ for some } \varepsilon>0 \implies u \in C^\alpha_{\loc}(\Omega).
		\end{equation}
	\end{itemize}

	We improve and extend these results in three ways:
	\begin{itemize}
		\item We sharpen the assumptions on $f$, proving that the conclusion of the higher differentiability result \eqref{eq:HDBL} holds under the weaker requirement that $f \in L^\frac{p}{p-1}(\Omega)$ and that the higher H\"older regularity result \eqref{eq:HHBLS} remains valid under the borderline assumption that $f$ belongs to the Marcinkiewicz space $L^{\frac{n}{sp-\alpha(p-1)},\infty}(\Omega)$. 
		\item We show that these results remain valid if the fractional $p$-Laplacian is perturbed by \emph{VMO coefficients}. In particular, we are able to gain a substantial amount of higher differentiability for equations with possibly discontinuous coefficients, which is not possible for local second-order elliptic equations with such general coefficients and thus a \emph{purely nonlocal phenomenon}, see Section \ref{sec:pr}.
		\item We additionally prove precise Sobolev regularity results for $f$ belonging to Lebesgue spaces, establishing a complete Calder\'on-Zygmund theory in $W^{t,q}$ spaces for the fractional $p$-Laplacian with VMO coefficients.
	\end{itemize}

	These fine Calder\'on-Zygmund-type estimates also extend and sharpen corresponding ones recently obtained by the second author in \cite{MeV,MeI} in the case when $p=2$, that is, in the context of linear nonlocal equations with VMO coefficients. The approach in \cite{MeV,MeI} relies on certain fractional gradients called dual pairs introduced in \cite{KMS1} and involved covering arguments inspired by \cite{caffaP,KMS1}. 

	The regularity results obtained in \cite{MeV,MeI} in turn generalize similar ones due to Mengesha, Schikorra and Yeepo established in \cite{MSY}, who assume that the coefficient $A$ is H\"older continuous and whose proof is predicated on commutator estimates inspired by classical ones from \cite{CRW}.
	
	We want to stress at this point that our approach departs substantially from the ones implemented in \cite{MSY,MeV,MeI} in the sense that it relies neither on commutator estimates nor on dual pairs. Instead, inspired by corresponding approaches in the context of local equations of $p$-Laplacian-type (see e.g.\ \cite{Iwaniec,KinnunenZhou,KMUniversal,DKSBMO,BCDKS,BDGP}) and some ideas introduced in \cite{KMS2,KMS18} to obtain fine zero-order regularity estimates in the nonlocal setting, we establish precise pointwise estimates in terms of certain \emph{fractional sharp maximal functions} in the spirit of classical harmonic analysis, see Theorem \ref{thm:fracmaxest}. 
	
	Since the relevant mapping properties of these maximal functions are known, this leads to the desired Calder\'on-Zygmund estimates in a considerably more streamlined fashion than in the mentioned previous works and enables us to additionally capture a large class of nonlinear nonlocal equations as well as some \emph{borderline cases} which are new already in the linear setting.

	\subsection{Setting}
	Before stating our main results, we need to introduce our setup in a more rigorous fashion. 
	
	For $\Lambda \geq 1,$ we define $\mathcal{L}_0(\Lambda)$ to be the class of all measurable coefficients $A:\mathbb{R}^n \times \mathbb{R}^n \to \mathbb{R}$ that satisfy the following conditions:
	\begin{itemize}
		\item There exists a constant  $\Lambda \geq 1$ such that 
		$$\Lambda^{-1} \leq A(x,y) \leq \Lambda \text{ for almost all } x,y \in \mathbb{R}^n .$$
		\item  $
		A(x,y)=A(y,x) \text{ for almost all } x,y \in \mathbb{R}^n.
		$
	\end{itemize}
	
	In order to control the growth of solutions at infinity, for $\beta >0$ and $q \in [1,\infty)$ we consider the tail spaces
	$$L^q_{\beta}(\mathbb{R}^n):= \left \{u \in L^1_{\loc}(\mathbb{R}^n) \mathrel{\Big|} \int_{\mathbb{R}^n} \frac{|u(y)|^q}{(1+|y|)^{n+\beta}}dy < \infty \right \}.$$
	
	We remark that a function $u \in L^1_{\loc}(\mathbb{R}^n)$ belongs to the space $L^{p-1}_{sp}(\mathbb{R}^n)$ if and only if the \emph{nonlocal tails} of $u$ given by
	$$ \textnormal{Tail}(u;x_0,R):=\left (R^{sp} \int_{\mathbb{R}^n \setminus B_{R}(x_0)} \frac{|u(y)|^{p-1}}{|x_0-y|^{n+sp}}dy \right )^\frac{1}{p-1}$$
	are finite for all $R>0$, $x_0 \in \mathbb{R}^n$.
	
	For notational convenience, we also consider the following nonlocal excess functional which was introduced in \cite{KMS2}.
	\begin{definition}[Nonlocal excess functional] \label{def:excess}
		Let $s \in (0,1)$ and $p \in [2,\infty)$. Given $x_0 \in \mathbb{R}^n$, $R>0$ and $u \in L^{p-1}_{sp}(\mathbb{R}^n)$, we define the nonlocal excess functional $E(u;x_0,R)$ by
		$$
			E(u;x_0,R):=\left (\dashint_{B_{R}(x_0)} |u-(u)_{B_R(x_0)}|^{p-1} dx \right )^\frac{1}{p-1} + \textnormal{Tail}(u-(u)_{B_R(x_0)};x_0,R).
		$$
	\end{definition}
	
	In addition, we use the notation
	\begin{equation} \label{eq:dexp}
	p^\star_s := \begin{cases} \normalfont
		\frac{np}{n-sp} & \text{ if } sp<n \\
		+\infty & \text{ if } sp>n,
	\end{cases} \quad \quad (p^\star_s)^\prime := \begin{cases} \normalfont
	\frac{np}{n(p-1)+sp} & \text{ if } sp<n \\
	1 & \text{ if } sp>n.
	\end{cases}
	\end{equation}
	
	We are now in the position to define weak solutions to \eqref{eq:splapeq} as follows.
	\begin{definition}[Weak solutions] \label{def:weaksol}
		Let $\Omega \subset \mathbb{R}^n$ be a bounded domain, $s \in (0,1)$, $p \in [2,\infty)$ and assume that $A \in \mathcal{L}_0(\Lambda)$ for some $\Lambda \geq 1$. Moreover, assume that $f \in L^{(p^\star_s)^\prime}(\Omega)$ if $sp \neq n$ and $f \in L^q(\Omega)$ for some $q>1$ if $sp=n$. We say that $u \in W^{s,p}(\Omega) \cap L^{p-1}_{sp}(\mathbb{R}^n)$ is a weak solution of \eqref{eq:splapeq}, if 
		\begin{equation} \label{eqweaksol}
			\begin{aligned}
			&\int_{\mathbb{R}^n} \int_{\mathbb{R}^n} \frac{A(x,y)}{|x-y|^{n+sp}} |u(x)-u(y)|^{p-2}(u(x)-u(y))(\varphi(x)-\varphi(y))dydx \\ &= \int_{\Omega} f \varphi dx \quad \forall \varphi \in C_0^\infty(\Omega).
			\end{aligned}
		\end{equation}
	\end{definition}
	For the definition of the standard fractional Sobolev spaces $W^{s,p}(\Omega)$, we refer to Section \ref{fracSob1}.
	
	In view of H\"older's inequality and the fractional Sobolev embedding, the assumptions from Definition \ref{def:weaksol} ensure that $f$ belongs to the dual of $W^{s,p}_0(\Omega)$, which guarantees the existence of weak solutions, see e.g.\ \cite[Proposition 2.12]{BLS}.
	
	Since we are concerned with proving higher regularity, it is natural to impose an additional regularity assumption on $A$. Namely, we require $A$ to be of vanishing mean oscillation or even to be merely small in BMO in the sense of the following definition.
	
	\begin{definition}[Small BMO and VMO coefficients]
		Let $\delta>0$, $\Lambda \geq 1$ and $A \in \mathcal{L}_0(\Lambda)$. 
		\begin{itemize}
			\item For $R>0$ and $x_0 \in \mathbb{R}^n$, we say that $A$ is $\delta$-vanishing in $B_R(x_0)$, if $$ \sup_{0<r \leq R} \, \dashint_{B_r(x_0)} \dashint_{B_r(x_0)} |A(x,y)-(A)_{r,x_0}|dydx \leq \delta ,$$
			where $(A)_{r,x_0}:= \dashint_{B_r(x_0)} \dashint_{B_r(x_0)} A(x,y)dydx$.
			\item For $R_0>0$, we say that $A$ is $(\delta,R_0)$-BMO in a domain $\Omega \subset \mathbb{R}^n$, if for any $x_0 \in \Omega$ and any $R \in (0,R_0]$ with $B_R(x_0) \subset \Omega$, $A$ is $\delta$-vanishing in $B_R(x_0)$. 
			\item We say that $A$ is VMO in $\Omega$, if for any $\delta>0$, there exists some $R_0>0$ such that $A$ is $(\delta,R_0)$-BMO in $\Omega$.
		\end{itemize}
	\end{definition}
	
	If $A$ belongs to the standard space $\textnormal{VMO}(\Omega \times \Omega)$ as defined in e.g.\ \cite[Section 2.1.1]{MPS} or \cite{Sarason}, then $A$ is also VMO in $\Omega$ in the above sense. However, our assumption essentially only requires $A$ to be of vanishing mean oscillation in some arbitrarily small open neighbourhood of the diagonal in $\Omega \times \Omega$, while away from the diagonal $A$ does not need to satisfy any regularity assumption at all.
	In analogy to classical VMO functions, if $A$ is continuous in an open neighbourhood of the diagonal in $\Omega \times \Omega$, then $A$ is also VMO in $\Omega$. 
	Nevertheless, continuity close to the diagonal is not essential, since there are plenty of discontinuous VMO functions, see e.g.\ \cite[Formula (1.8) and (1.9)]{MeV}.
	
	\subsection{Main results} \label{sec:mr}
	
	We establish the following Sobolev regularity result for equations of fractional $p$-Laplace-type with coefficients that are sufficiently small in BMO.
	
	\begin{theorem}[Calder\'on-Zygmund estimates] \label{thm:CZestBMO}
		Let $\Omega \subset \mathbb{R}^n$ be a bounded domain, $s \in (0,1)$, $p \in [2,\infty)$, $R_0>0$ and $\Lambda \geq 1$. In addition, let $f$ be as in Definition \ref{def:weaksol} and fix some 
		\begin{equation} \label{eq:trange}
		 t \in \bigg [s, \min \left \{\frac{sp}{p-1},1 \right \} \bigg ).
		\end{equation}
		There exists some small $\delta=\delta(n,s,p,\Lambda,t)>0$ such that if $A \in \mathcal{L}_0(\Lambda)$ is $(\delta,R_0)$-BMO in $\Omega$, 
		then for any weak solution $u \in W^{s,p}(\Omega) \cap L^{p-1}_{sp}(\mathbb{R}^n)$
		of \eqref{eq:splapeq}, any $q \in \big (1,\frac{n}{sp-t(p-1)} \big)$ and $q^\star:=\frac{nq(p-1)}{n-(sp-t(p-1))q}$,
		we have the implication 
		\begin{equation} \label{eq:CZregBMO}
			f \in L^q(\Omega) \implies u \in
				W^{t,q^\star}_{\loc}(\Omega).
		\end{equation}
		Moreover, for any $R >0$ and any $x_0 \in \Omega$ such that $B_R(x_0) \subset \Omega$, we have
		\begin{equation} \label{eq:CZest}
			\begin{aligned}
				[u]_{W^{t,q^\star}(B_{R/2}(x_0))} \leq C \left ( R^{n/q^\star-t} E(u;x_0,R) + \norm{f}_{L^q(B_R(x_0))}^\frac{1}{p-1} \right ),
			\end{aligned}
		\end{equation}
		where $C$ depends only on $n,s,p,\Lambda,t,q,R_0$ and $\max \{\textnormal{diam}(\Omega),1\}$.
	\end{theorem}
	
	\begin{remark}[Sobolev regularity under VMO coefficients] \label{rem:VMO} \normalfont
	A straightforward consequence of Theorem \ref{thm:CZestBMO} is that if we impose the slightly stronger assumption that $A$ is VMO in $\Omega$, then in the setting of Theorem \ref{thm:CZestBMO} the implication \eqref{eq:CZregBMO} holds for all $t$ in the range $s \leq t < \min \big \{\frac{sp}{p-1},1 \big \}$.
	\end{remark} 

	\begin{remark}[SOLA] \label{rem:SOLA} \normalfont
		If the exponent $q>1$ is smaller than the duality exponent $(p^\star_s)^\prime$ as defined in \eqref{eq:dexp}, then the requirements on $f$ from Definition \ref{def:weaksol} are stronger than the one imposed in \eqref{eq:CZregBMO}. In this case it is more natural to consider a more general notion of solutions called SOLA (=\,solutions obtained by limiting approximations) as defined in \cite[Definition 1]{KMS2}, which are essentially distributional solutions that can be approximated by weak solutions of corresponding regularized problems. Such solutions are known to exist already when $f \in L^1(\Omega)$, see \cite[Theorem 1.1]{KMS2}.
		While for simplicity we restrict ourselves to standard weak solutions, we remark that by combining Theorem \ref{thm:CZestBMO} with a compactness argument similar to the one in \cite[Corollary 4.5]{KNS}, it is straightforward to show that the implication \eqref{eq:CZregBMO} and the estimate \eqref{eq:CZest} remain valid for such SOLA instead of weak solutions.
	\end{remark} 
	
	Next, we have the following higher H\"older regularity result under sharp borderline assumptions on $f$.
	\begin{theorem}[Fine higher H\"older regularity] \label{thm:fineHold}
	Let $\Omega \subset \mathbb{R}^n$ be a bounded domain, $s \in (0,1)$, $p \in [2,\infty)$, $R_0>0$ and $\Lambda \geq 1$. In addition, let $f$ be as in Definition \ref{def:weaksol} and fix some \begin{equation} \label{eq:trangex}
		\alpha \in \left (0, \min \left \{\frac{sp}{p-1},1 \right \} \right )
	\end{equation} such that $sp-\alpha(p-1)<n$. There exists some small $\delta=\delta(n,s,p,\Lambda,\alpha)>0$, such that if $A \in \mathcal{L}_0(\Lambda)$ is $(\delta,R_0)$-BMO in $\Omega$, then for any weak solution $u \in W^{s,p}(\Omega) \cap L^{p-1}_{sp}(\mathbb{R}^n)$
	of \eqref{eq:splapeq}, we have the implication 
	\begin{equation} \label{eq:fineHold}
		f \in L^{\frac{n}{sp-\alpha(p-1)},\infty}(\Omega) \implies u \in C^\alpha_{\loc}(\Omega).
	\end{equation}
	Moreover, for any $R >0$ and any $x_0 \in \Omega$ such that $B_R(x_0) \subset \Omega$, we have
	\begin{equation} \label{eq:FineHoldest}
			[u]_{C^\alpha(B_{R/2}(x_0))} \leq C \left (R^{-\alpha} E(u;x_0,R) + \norm{f}_{L^{\frac{n}{sp-\alpha(p-1)},\infty}(B_R(x_0))}^\frac{1}{p-1} \right ) ,
	\end{equation}
	where $C$ depends only on $n,s,p,\Lambda,\alpha,R_0$ and $\max \{\textnormal{diam}(\Omega),1\}$.
	\end{theorem}

	Here $L^{\frac{n}{sp-\alpha(p-1)},\infty}(\Omega)$ denotes the standard Marcinkiewicz space of order $\frac{n}{sp-\alpha(p-1)}$, see Definition \ref{def:MS}. Since we have the proper inclusion $L^{\frac{n}{sp-\alpha(p-1)}}(\Omega) \subsetneq L^{\frac{n}{sp-\alpha(p-1)},\infty}(\Omega)$, Theorem \ref{thm:fineHold} can indeed be considered to be a regularity result of \emph{borderline flavour}.

	\begin{remark}[H\"older regularity under VMO coefficients] \normalfont
		In analogy to Remark \ref{rem:VMO}, Theorem \ref{thm:fineHold} clearly implies that if we require $A$ to be VMO in $\Omega$, then in the setting of Theorem \ref{thm:fineHold} the implication \eqref{eq:FineHoldest} holds for all $\alpha$ in the range $0< \alpha < \min \big \{\frac{sp}{p-1},1 \big \}$.
	\end{remark}

	\begin{remark}[Sharpness] \label{rem:sharp} \normalfont
		If $s \leq \frac{p-1}{p}$, then Theorem \ref{thm:fineHold} is sharp already in the case $A \equiv 1$, see \cite[Example 1.5]{BLS}.
	\end{remark}

	We also record the following pure higher differentiability result.
	\begin{corollary}[Pure higher differentiability] \label{cor:HD}
		Let $\Omega \subset \mathbb{R}^n$ be a bounded domain, $s \in (0,1)$, $p \in [2,\infty)$, $\Lambda \geq 1$ and assume that $A \in \mathcal{L}_0(\Lambda)$ is VMO in $\Omega$. 
		Then for any weak solution $u \in W^{s,p}(\Omega) \cap L^{p-1}_{sp}(\mathbb{R}^n)$
		of $
		(-\Delta )^s_{p,A} u = f \text{ in } \Omega,
		$
		we have the implication
		\begin{equation} \label{eq:HDVMO}
			f \in L^\frac{p}{p-1}(\Omega) \implies u \in W^{t,p}_{\loc}(\Omega) \text{ for any } t \in \left (s, \min \left \{\frac{sp}{p-1},1 \right \} \right ).
		\end{equation}
	\end{corollary}

	As mentioned, all of the above regularity results are in fact consequences of precise pointwise estimates in terms of certain \emph{nonlocal fractional sharp maximal functions}, see Section \ref{sec:pmf} and in particular Theorem \ref{thm:fracmaxest}.
	
	\subsection{Differential stability effect and related results} \label{sec:pr}
	Since formally nonlocal operators of the type \eqref{eq:splap} converge to corresponding local elliptic operators as $s \to 1$ (see \cite{BBM1,FKV} for some rigorous results in this direction), our main results can be seen as counterparts to Calder\'on-Zygmund-type estimates for local equations of $p$-Laplacian-type obtained in e.g.\ \cite{Iwaniec,caffaP,KinnunenZhou,DM2,Cianchi,KMUniversal,CMJEMS,BDM20,DFJMPA,BDGP}. \par
	It is noteworthy that from the point of view of such local second-order elliptic equations, the differentiability gain in Theorem \ref{thm:CZestBMO} and Corollary \ref{cor:HD} in the presence of such general coefficients can be considered to be somewhat surprising. In fact, already in the case of linear second-order elliptic equations $\text{div}(B \nabla u)=f$, no differentiability gain at all is attainable without imposing a corresponding amount of differentiability on the coefficients, see \cite[Section 1]{KMS1}. In contrast, Theorem \ref{thm:CZestBMO} and Corollary \ref{cor:HD} yield a substantial amount of higher differentiability already in the presence of possibly discontinuous coefficients of VMO-type.
	
	Such \emph{nonlocal differential stability effects} were first observed in the papers \cite{KMS1,SchikorraMA}, where it was proved that in our nonlocal setting a slight gain of differentiability is already possible under merely bounded measurable coefficients. In the linear case, the subsequent works \cite{MSY,FMSYPDEA,MeV,MeI} then demonstrated that this effect becomes enhanced if one additionally imposes continuity or VMO assumptions on $A$. Our main results show that these enhanced nonlocal differential stability effects persist even in the nonlinear and degenerate setting of fractional $p$-Laplacian-type equations.
	More precisely, as hinted at in Section \ref{sec:overview}, in the linear case when $p=2$ Theorem \ref{thm:CZestBMO} recovers corresponding estimates from \cite[Theorem 1.2]{MeI} and extends them to the nonlinear case when $p>2$, in which case Theorem \ref{thm:CZestBMO} is a completely new result already in the case $A \equiv 1$ of the standard fractional $p$-Laplacian. 
	
	Moreover, our higher H\"older regularity result Theorem \ref{thm:fineHold} is already new in the linear case when $p=2$, where it sharpens \cite[Theorem 1.4]{MeV} as well as \cite[Theorem 1.3]{FallDCDS}. In fact, in \cite[Theorem 1.4]{MeV}, the slightly stronger assumption that $f \in L^{\frac{n}{2s-\alpha}}(\Omega)$ instead of $f \in L^{\frac{n}{2s-\alpha},\infty}(\Omega)$ was required to conclude that $u \in C^\alpha_{\loc}(\Omega)$, while in \cite[Theorem 1.3]{FallDCDS} the implication \eqref{eq:fineHold} for $p=2$ was proved under the stronger assumption that $A$ is continuous.
	
	As indicated in Section \ref{sec:overview}, in the case $A \equiv 1$ of the standard fractional $p$-Laplacian, Theorem \ref{thm:fineHold} also sharpens the higher H\"older regularity result \cite[Theorem 1.4]{BLS}, in which $f \in L^{\frac{n}{sp-\alpha(p-1)}+\varepsilon}(\Omega)$ for some $\varepsilon>0$ was required in order to conclude that $u \in C^\alpha_{\loc}(\Omega).$ In addition, Corollary \ref{cor:HD} improves the higher differentiability result \cite[Theorem 1.5]{BL} in the case when $s \leq \frac{p-1}{p}$, where $f \in W^{s,\frac{p}{p-1}}(\Omega)$ instead of $f \in L^\frac{p}{p-1}(\Omega)$ was assumed to obtain higher differentiability in the range $t \in \big (s, \frac{sp}{p-1}\big )$.
	
	Furthermore, in the case of general bounded measurable coefficients $A \in \mathcal{L}_0(\Lambda)$, in \cite{DKP,CozziJFA} it was proved by De Giorgi-Nash-Moser-type methods that weak solutions to $(-\Delta)^s_{p,A} u=0$ are locally H\"older continuous with respect to some small not explicitly quantified exponent. Under the mild additional assumption that $A$ is small in BMO, Theorem \ref{thm:fineHold} improves this unspecified H\"older exponent to an explicit one, which as discussed in Remark \ref{rem:sharp} is sharp in the case when $s \leq \frac{p-1}{p}$.
	
	Without being exhaustive at all, we note that more related results concerning H\"older regularity for various kinds of nonlocal equations can for instance be found in \cite{Silvestre,KassCalcVar,CCV,CSARMA,BCF,DKP2,BP,DFPJDE,FallCalcVar,MeH,CKWCalcVar,BKJMA}. Moreover, some further results concerning Sobolev regularity were for example proved in \cite{DongKim,Grubb,CozziSob,NowakNA,MengeshaScott,ma_yang_2022,DLCVPDE}.
	
	\subsection{Some open questions} \label{sec:op}
	In order to facilitate further research concerning nonlocal equations of fractional $p$-Laplacian-type, let us discuss some natural open questions related to our main results.
	
	\textbf{Gradient estimates}: In the case when $s > \frac{p-1}{p}$, in view of the classical $C^{1,\alpha}$ estimates for the local $p$-Laplacian (see e.g.\ \cite{NU68,KU77}) one might expect $(s,p)$-harmonic functions, that is, weak solutions to \eqref{eq:splapeq} in the case when $A \equiv 1$ and $f=0$, to have H\"older continuous gradients, which is so far unknown. 

	In case of a positive breakthrough in this direction, further natural questions would for example be to which extend such gradient estimates remain valid in the presence of a general right-hand side and under more general coefficients such as H\"older continuous ones. A question of particular interest in this context would concern the possibility of obtaining fine pointwise gradient estimates as it was accomplished in \cite{KuMiARMA1,KuMi} for the local $p$-Laplacian in terms of Riesz potentials of the data and in \cite{KMUniversal,DKSBMO,BCDKS} in terms of sharp maximal functions. We note that such gradient potential estimates were recently established in \cite{KNS} also for a large class of linear nonlocal equations.
	
	\textbf{Boundary regularity}: While in the case of the local $p$-Laplacian the interior regularity in most cases remains valid up to the boundary if the domain is regular enough (see e.g.\ \cite{ByunWang07,BCDS,BBDL}), the boundary regularity in the nonlocal case is rather restricted already in the case of the fractional Laplacian, see \cite{RSJMPA,AFLY}. Moreover, for the fractional $p$-Laplacian so far only global H\"older regularity with respect to some small unspecified exponent seems to be known, see \cite{IMS}. Thus, a natural question is to which degree our main results remain valid up to the boundary.
	
	\textbf{Parabolic equations}: Another naturally arising question is if versions of our main results can be proved for parabolic counterparts of the equation \eqref{eq:splapeq} as considered in e.g.\ \cite{MRT,VazquezNA,BLS21,DZSCVPDE,Liao,AT23}. As in the case of the local parabolic $p$-Laplacian (see e.g.\ \cite{AcerbiMingione}), this would entail handling the anisotropic structure of the equation, which is usually accomplished by considering intrinsic space/time cylinders whose size depends on the solution itself. 
	
	\section{Preliminaries}
	
	\subsection{Notation} \label{notation}
	For the sake of convenience, let us fix some general notation which we use throughout the paper. By $C$ we denote a general positive constant which possibly varies from line to line and only depends on the parameters indicated in the statement to be proved unless specified otherwise.
	As usual, by
	$$ B_r(x_0):= \{x \in \mathbb{R}^n \mid |x-x_0|<r \}$$
	we denote the open euclidean ball with center $x_0 \in \mathbb{R}^n$ and radius $r>0$. \par 
	Furthermore, if $E \subset \mathbb{R}^n$ is measurable, then by $|E|$ we denote the $n$-dimensional Lebesgue-measure of $E$. If $0<|E|<\infty$, then for any $u \in L^1(E)$ we define
	$$ (u)_E:= \dashint_{E} u(x)dx := \frac{1}{|E|} \int_{E} u(x)dx.$$
	
\subsection{$W^{s,p}$ spaces} \label{fracSob1}
\begin{definition}
	Let $\Omega \subset \mathbb{R}^n$ be a domain. For $p \in [1,\infty)$ and $s \in (0,1)$, we define the fractional Sobolev space
	$$W^{s,p}(\Omega):=\left \{u \in L^p(\Omega) \mid [u]_{W^{s,p}(\Omega)}<\infty \right \}$$
	with norm
	$$ \norm{u}_{W^{s,p}(\Omega)} := \left (\norm{u}_{L^p(\Omega)}^p + [u]_{W^{s,p}(\Omega)}^p \right )^{1/p} ,$$
	where
	$$ [u]_{W^{s,p}(\Omega)}:=\left (\int_{\Omega} \int_{\Omega} \frac{|u(x)-u(y)|^p}{|x-y|^{n+sp}}dydx \right )^{1/p} .$$
	Moreover, we define the corresponding local fractional Sobolev spaces by
	$$ W^{s,p}_{\loc}(\Omega):= \left \{ u \in L^p_{\loc}(\Omega) \mid u \in W^{s,p}(\Omega^\prime) \text{ for any domain } \Omega^\prime \Subset \Omega \right \}.$$
	Also, we define the space 
	$$W^{s,p}_0(\Omega):= \left \{u \in W^{s,p}(\mathbb{R}^n) \mid u = 0 \text{ in } \mathbb{R}^n \setminus \Omega \right \}.$$
\end{definition}

\subsection{Fractional maximal functions} \label{Cald}

\begin{definition}
Let $x_0 \in \mathbb{R}^n$. Given $g \in L^1_{\loc}(\mathbb{R}^n)$, we define the fractional maximal function of order $\beta \in [0,n]$ by
	\begin{equation}
		{M}_{\beta} (g)(x_0):= \sup_{r>0} \, r^\beta \dashint_{B_r(x_0)} |g|dx.
	\end{equation}
Moreover, given $R>0$ and $g \in L^1(B_R(x_0))$, we also consider the localized fractional maximal function of order $\beta \in [0,n]$ by
\begin{equation}
	{M}_{R,\beta} (g)(x_0):= \sup_{0<r\leq R} \, r^\beta \dashint_{B_r(x_0)} |g|dx.
\end{equation}
\end{definition}

The fractional maximal function has the following well-known mapping properties in $L^q$ spaces, see e.g.\ \cite{AdamsDuke,KinSak}.
\begin{proposition} \label{prop:fracmaxest}
	Let $\beta \in [0,n)$, $q \in (1,n/\beta)$ and $g \in L^q(\mathbb{R}^n)$. Then we have the estimate
	$$ \norm{{M}_{\beta} (g)}_{L^\frac{nq}{n-\beta q}(\mathbb{R}^n)} \leq C \norm{g}_{L^q(\mathbb{R}^n)},$$
	where $C$ depends only on $n,\beta,q$.
\end{proposition}

In order to formulate the sharp mapping properties of the fractional maximal function with respect to $L^\infty$, we need to recall the following definition of Marcinkiewicz spaces which also appear in our higher H\"older regularity result Theorem \ref{thm:fineHold}.
\begin{definition} \label{def:MS}
	For any domain $\Omega \subset \mathbb{R}^n$ and any $q \in [1,\infty)$, we define the \emph{Marcinkiewicz space} $L^{q,\infty}(\Omega)$ as the space of measurable functions $f: \Omega \to \mathbb{R}^n$ such that the quasinorm
	$$ \norm{f}_{L^{q,\infty}(\Omega)}:=
		\sup_{t>0} t \hspace{0.2mm} | \{x \in \Omega \mid |f(x)| \geq t \} |^\frac{1}{q} $$
	is finite.
\end{definition}

In view of Chebychev's inequality, for any $q \in [1,\infty)$ we clearly have the continuous embedding
$L^{q}(\Omega) \hookrightarrow L^{q,\infty}(\Omega)$. Since for $x_0 \in \Omega$ the function $x \mapsto |x-x_0|^{-n/q}$ clearly belongs to $L^{q,\infty}(\Omega) \setminus L^{q}(\Omega)$, the inclusion is proper.

The well-known mapping properties of $\mathcal{M}_\beta$ with respect to $L^\infty$ are then given as follows, see e.g.\ \cite{KuMiG}. We include a simple proof for the sake of completeness.

\begin{proposition} \label{lem:Linfbd}
	Let $\beta \in (0,n)$ and $g \in L^{\frac{n}{\beta},\infty}(\mathbb{R}^n)$. Then we have the estimate
	\begin{equation} \label{eq:supfm}
	\sup_{x_0 \in \mathbb{R}^n} {M}_{\beta} (g)(x_0) \leq C \norm{g}_{L^{\frac{n}{\beta},\infty}(\mathbb{R}^n)},
	\end{equation}
	where $C$ depends only on $n,\beta$.
\end{proposition}

\begin{proof}
	If $g \equiv 0$ a.e.\ in $\mathbb{R}^n$, then \eqref{eq:supfm} is trivially satisfied, so that we can assume that $\norm{g}_{L^{\frac{n}{\beta},\infty}(\mathbb{R}^n)}>0$.
	Next, set $\widehat g:= g/\norm{g}_{L^{\frac{n}{\beta},\infty}(\mathbb{R}^n)},$ so that $\norm{\widehat g}_{L^{\frac{n}{\beta},\infty}(\mathbb{R}^n)} =1$.
	
	Fix some $x_0 \in \mathbb{R}^n$ and denote by $\omega_n$ the Lebesgue measure of the $n$-dimensional unit ball. Using Cavalieri's principle, for any $r>0$ we estimate
	\begin{align*}
		 r^\beta \dashint_{B_r(x_0)} |\widehat g|dx & = r^{\beta} |B_r(x_0)|^{-1} \int_0^\infty | \{x \in B_r(x_0) \mid |\widehat g(x)|>\lambda \} | d\lambda \\
		 & \leq r^{\beta} \int_0^{r^{-\beta}} d\lambda
		 + \omega_n^{-1} r^{\beta-n} \int_{r^{-\beta}}^\infty | \{x \in B_r(x_0) \mid |\widehat g(x)|>\lambda \} | d\lambda \\
		 & \leq 1 + \norm{\widehat g}_{L^{\frac{n}{\beta},\infty}(\mathbb{R}^n)}^{n/\beta} \omega_n^{-1} r^{\beta-n} \int_{r^{-\beta}}^\infty \lambda^{-\frac{n}{\beta}} d\lambda = 1 + \frac{\beta}{\omega_n(n-\beta)},
	\end{align*}
	which after rescaling yields the desired estimate with respect to $C=1 + \frac{\beta}{\omega_n(n-\beta)}$.
\end{proof}

In order to control the oscillations rather than the size of functions, we also use the following notion of fractional sharp maximal functions considered e.g.\ in \cite{DeVoreSharpley}.

\begin{definition}[Fractional sharp maximal function]
	Given $x_0 \in \mathbb{R}^n$, $R>0$ and $g \in L^{1}(B_R(x_0))$, we define the fractional sharp maximal function of order $\beta \in [0,1]$ by
	\begin{equation}
	{M}^\#_{R,\beta} (g)(x_0):= \sup_{0<r\leq R} \, r^{-\beta} \dashint_{B_r(x_0)} |g-(g)_{B_r(x_0)}|dx.
	\end{equation}
\end{definition}

The following estimate follows from \cite[Proposition 1]{KuMiG}.

\begin{proposition} \label{prop:sfm}
	Let $R>0$ and $g \in L^{1}(B_{2R})$. Then for any $\beta \in (0,1]$ and all $x,y \in B_{R/4}$, we have the inequality
	$$ |g(x)-g(y)| \leq \frac{C}{\beta} \left [{M}^\#_{R,\beta}(g)(x) + {M}^\#_{R,\beta}(g)(y) \right ] |x-y|^\beta,$$
	where $C$ depends only on $n$.
\end{proposition}

The following consequence of Proposition \ref{prop:sfm} allows us to obtain estimates in fractional Sobolev spaces $W^{s,p}$ by proving appropriate estimates in terms of fractional sharp maximal functions.

\begin{proposition} \label{prop:embedding}
	Let $R>0$, $1 < \widetilde q < q < \infty$ and $s,\widetilde s \in (0,1)$. If $ s - n/p = \widetilde s - n/\widetilde p, $ then for any $g \in L^{1}(B_{2R})$ we have the estimates
	$$ \norm{g}_{W^{s,q}(B_{R/8})} \leq \widetilde C \norm{g}_{N^{\widetilde s,\widetilde q}(B_{R/8})} \leq C \left [ \norm{g}_{L^{\widetilde q}(B_{R/8})} + \norm{{M}^\#_{R,\widetilde s}(g)}_{L^{\widetilde q}(B_R)} \right ],$$
	where $C$ and $\widetilde C$ depend only on $n,s,\widetilde s,q,\widetilde q,R$ and the Nikolskii norm $\norm{g}_{N^{\widetilde s,\widetilde q}(B_{R/8})}$ is defined by 
	\begin{equation} \label{eq:Nikolskii}
	\norm{g}_{N^{\widetilde s,\widetilde q}(B_{R/8})} := \norm{g}_{L^{\widetilde q}(B_{R/8})} + \sup_{0<|h|<R/8} \left (\int_{B_{R/8}} \frac{|g(x+h)-g(x)|^{\widetilde q}}{|h|^{\widetilde s \widetilde q}}dx \right )^\frac{1}{\widetilde q}.
	\end{equation}
\end{proposition}

\begin{proof}
	From Proposition \ref{prop:sfm} it follows that
	\begin{align*}
	& \sup_{0<|h|<R/8} \left (\int_{B_{R/8}} \frac{|g(x+h)-g(x)|^{\widetilde q}}{|h|^{\widetilde s \widetilde q}}dx \right )^\frac{1}{\widetilde q} \\ & \leq C \left [ \sup_{0<|h|<R/8} \left (\int_{B_{R/8}} |{M}^\#_{R,\widetilde s}(g)(x+h)|^{\widetilde q}dx \right )^\frac{1}{\widetilde q} + \left (\int_{B_{R/8}} |{M}^\#_{R,\widetilde s}(g)(x)|^{\widetilde q}dx \right )^\frac{1}{\widetilde q} \right ] \\
	& \leq C \norm{{M}^\#_{R,\widetilde s}(g)}_{L^{\widetilde q}(B_{R})}. 
	\end{align*}
	Since e.g.\ by \cite[Remark 2]{CozziSob} \eqref{eq:Nikolskii} indeed gives rise to an equivalent norm in the standard Nikolskii space $N^{\widetilde s,\widetilde q}(B_{R/8})$ and by \cite[Proposition 4]{CozziSob} we have the continuous embedding $N^{\widetilde s,\widetilde q}(B_{R/8}) \hookrightarrow W^{s,q}(B_{R/8}),$ we conclude that
	\begin{align*}
		\norm{g}_{W^{s,q}(B_{R/8})} \leq \widetilde C \norm{g}_{N^{\widetilde s,\widetilde q}(B_{R/8})} \leq C \left [ \norm{g}_{L^{\widetilde q}(B_{R/8})} + \norm{{M}^\#_{R,\widetilde s}(g)}_{L^{\widetilde q}(B_R)} \right ],
	\end{align*}
	finishing the proof.
\end{proof}
	
	\subsection{Preliminary regularity results}
	
	\begin{remark} \label{rem:u-c} \normalfont
	Since our aim is to control the oscillations of solutions, a simple fact that we shall regularly use is that for any weak solution $u \in W^{s,p}(\Omega) \cap L^{p-1}_{sp}(\mathbb{R}^n)$ of $(-\Delta )^s_{p,A} u = f$ in a bounded domain $\Omega \subset \mathbb{R}^n$ and any constant $c \in \mathbb{R}$, the function $\widetilde u:=u-c \in W^{s,p}(\Omega) \cap L^{p-1}_{sp}(\mathbb{R}^n)$ is again a weak solution of $(-\Delta )^s_{p,A} \widetilde u = f$ in $\Omega$.
	\end{remark}

	\begin{remark} \label{rem:infc} \normalfont
	Another useful observation is that for all $R>0$, $x_0 \in \mathbb{R}^n$ and any $u \in L^{p-1}_{sp}(\mathbb{R}^n)$, the nonlocal excess $E(u;x_0,R)$ is comparable to the quantity
	$$ \widetilde E(u;x_0,R):= \inf_{c \in \mathbb{R}} \left [ \left (\dashint_{B_{R}(x_0)} |u-c|^{p-1} dx \right )^\frac{1}{p-1} + \textnormal{Tail}(u-c;x_0,R) \right ].$$
	In fact, in view of H\"older's inequality we clearly have
	$$ \widetilde E(u;x_0,R) \leq E(u;x_0,R) \leq C(n,s,p) \widetilde E(u;x_0,R).$$
	\end{remark}
	
	\subsubsection{Caccioppoli inequality}
	
	The following Caccioppoli-type inequality follows by combining \cite[Lemma 2.2]{KMS2} with Remark \ref{rem:u-c}.
	\begin{proposition} \label{prop:cacc}
		Let $s \in (0,1)$, $p \in [2,\infty)$, $R>0$, $x_0 \in \mathbb{R}^n$, $\Lambda \geq 1$ and fix some $A \in \mathcal{L}_0(\Lambda)$. 
		Moreover, assume that $u \in W^{s,p}(B_{R}(x_0)) \cap L^{p-1}_{sp}(\mathbb{R}^n)$ is a weak solution of the equation $(-\Delta )^s_{p,A} u = 0$ in $B_{R}(x_0)$. Then we have the estimate
		\begin{equation} \label{eq:cacc}
				\left (\dashint_{B_{R/2}(x_0)} \int_{B_{R/2}(x_0)} \frac{|u(x)-u(y)|^p}{|x-y|^{n+sp}} dydx \right )^\frac{1}{p} \leq C R^{-s} E(u;x_0,R),
		\end{equation}
		where $C=C(n,s,p,\Lambda)>0$.
	\end{proposition}
	
	\subsubsection{Local boundedness}
	
	The following result on local boundedness follows by combining \cite[Corollary 2.1]{KMS2} with Remark \ref{rem:u-c}.
	\begin{proposition} \label{prop:lb}
		Let $s \in (0,1)$, $p \in [2,\infty)$, $R>0$, $x_0 \in \mathbb{R}^n$, $\Lambda \geq 1$ and fix some $A \in \mathcal{L}_0(\Lambda)$. 
		Moreover, assume that $u \in W^{s,p}(B_{R}(x_0)) \cap L^{p-1}_{sp}(\mathbb{R}^n)$ is a weak solution of the equation $(-\Delta )^s_{p,A} u = 0$ in $B_{R}(x_0)$. Then we have the estimate
		\begin{equation} \label{eq:lb}
				\norm{u-(u)_{B_R(x_0)}}_{L^\infty(B_{R/2}(x_0))} \leq C E(u;x_0,R),
		\end{equation}
		where $C=C(n,s,p,\Lambda)>0$.
	\end{proposition}
	
	\subsubsection{Higher H\"older regularity under locally constant coefficients}
	
	The following result is concerned with higher H\"older regularity in the case when the coefficient $A$ is locally constant and is a rather straightforward consequence of the H\"older estimates for the standard fractional $p$-Laplacian from \cite{BLS}.
	\begin{proposition} \label{prop:Holdreg}
		Let $s \in (0,1)$, $p \in [2,\infty)$, $R>0$, $x_0 \in \mathbb{R}^n$, $\Lambda \geq 1$ and fix some $A \in \mathcal{L}_0(\Lambda)$. In addition, fix some constant $b \in [\Lambda^{-1},\Lambda]$ and consider the locally frozen coefficient 
		$$ \widetilde A (x,y) := \begin{cases} \normalfont
			b & \text{if } (x,y) \in B_{R}(x_0) \times B_{R}(x_0) \\
			A(x,y) & \text{if } (x,y) \notin B_{R}(x_0) \times B_{R}(x_0).
		\end{cases} $$
		Moreover, assume that $u \in W^{s,p}(B_{R}(x_0)) \cap L^{p-1}_{sp}(\mathbb{R}^n)$ is a weak solution of the equation $(-\Delta )^s_{p,\widetilde A} u = 0$ in $B_{R}(x_0)$. Then for any $0<\alpha<\min \big \{\frac{sp}{p-1},1 \big\}$ and any $\rho \in (0,1/4]$, we have the oscillation estimate
		\begin{equation} \label{eq:Holdest}
				\osc_{B_{\rho R}(x_0)} u \leq C \rho^\alpha E(u;x_0,R),
		\end{equation}
		where $C=C(n,s,p,\Lambda,\alpha)>0$.
	\end{proposition}
	
	\begin{proof}
		Taking into account Remark \ref{rem:u-c}, observe that $\widetilde u:=u-(u)_{B_R(x_0)}$ is a weak solution of the inhomogeneous fractional $p$-Laplace equation $(-\Delta )^s_{p} u = f$ in $B_{R/2}(x_0)$, where 
		$$f(x):= b^{-1} \int_{\mathbb{R}^n \setminus B_{R}(x_0)} (b-A(x,y)) \frac{|\widetilde u(x)-\widetilde u(y)|^{p-2}(\widetilde u(x)-\widetilde u(y))}{|x-y|^{n+sp}}dy$$
		belongs to $L^\infty(B_{R/2}(x_0))$. In fact, by Proposition \ref{prop:lb} we have
		\begin{align*}
			\norm{f}_{L^\infty(B_{R/2}(x_0))} & \leq C \int_{\mathbb{R}^n \setminus B_{R}(x_0)} \frac{\norm{\widetilde u}_{L^\infty(B_{R/2}(x_0))}^{p-1}+|\widetilde u(y)|^{p-1}}{|x_0-y|^{n+sp}}dy \\
			& = C \Bigg (R^{-sp}\norm{\widetilde u}_{L^\infty(B_{R/2}(x_0))}^{p-1} + \int_{\mathbb{R}^n \setminus B_{R}(x_0)} \frac{|\widetilde u(y)|^{p-1}}{|x_0-y|^{n+sp}}dy \Bigg ) \\
			& \leq C \Bigg (R^{-sp} \dashint_{B_{R}(x_0)} |\widetilde u|^{p-1}dx  + \int_{\mathbb{R}^n \setminus B_{R}(x_0)} \frac{|\widetilde u(y)|^{p-1}}{|x_0-y|^{n+sp}}dy \Bigg ).
		\end{align*}
		Together with the H\"older estimates given by \cite[Theorem 1.4]{BLS} and Proposition \ref{prop:lb}, we obtain
		\begin{align*}
			& [u]_{C^\alpha(B_{R/4}(x_0))} = [\widetilde u]_{C^\alpha(B_{R/4}(x_0))} \\ & \leq \frac{C}{R^\alpha} \Bigg [ \norm{\widetilde u}_{L^\infty(B_{R/2}(x_0))} + \textnormal{Tail}(\widetilde u;x_0,R/2) + \left (R^{sp}\, \norm{f}_{L^\infty(B_{R/2}(x_0))} \right )^\frac{1}{p-1} \Bigg ] \\ & \leq \frac{C}{R^\alpha} E(u;x_0,R),
		\end{align*}
		which implies the oscillation estimate \eqref{eq:Holdest}.
	\end{proof}
	
	\subsubsection{Sobolev regularity for small exponents}
	
	We have the following Gehring-type result which is essentially proved in \cite[Theorem 1.3]{SchikorraMA} and by an alternative approach in \cite[Theorem 1.6 and Theorem 5.1]{MengeshaScott} and the companion note \cite{MengeshaScott1}. The latter proof builds on covering arguments in terms of a special type of fractional gradients called dual pairs introduced in \cite{KMS1}.
	
	\begin{proposition} \label{prop:Gehring}
		Let $s \in (0,1)$, $p \in [2,\infty)$, $r>0$, $x_0 \in \mathbb{R}^n$ and $\Lambda \geq 1$. If $A$ belongs to the class $\mathcal{L}_0(\Lambda)$, then there exists some small enough $\tau=\tau(n,s,p,\Lambda) \in (0,1-s)$ such that
		for any weak solution $u \in W^{s,p}(B_r(x_0)) \cap L^{p-1}_{sp}(\mathbb{R}^n)$
		of the equation
		$
		(-\Delta )^s_{p,A} u = 0 \text{ in } B_{r}(x_0)$ and any $c \in \mathbb{R}$, we have the estimate
		\begin{equation} \label{eq:HDE}
			\begin{aligned}
				& \left (\dashint_{B_{r/2}(x_0)} \int_{B_{r/2}(x_0)} \frac{|u(x)-u(y)|^{p+\tau}}{|x-y|^{n+(s+\tau)(p+\tau)}} dydx \right )^{\frac{1}{p+\tau}} \\ & \leq \frac{C}{r^\tau} \Bigg [\left (\dashint_{B_{r}(x_0)} \int_{B_{r}(x_0)} \frac{|u(x)-u(y)|^{p}}{|x-y|^{n+sp}} dydx \right )^{\frac{1}{p}} + \textnormal{Tail}(u-c;x_0,r/2) \Bigg ],
			\end{aligned}
		\end{equation}
		where $C$ depends only on $n,s,p,\Lambda$.
	\end{proposition}
	
	Since strictly speaking in both papers \cite{SchikorraMA,MengeshaScott} only global versions of the above result are proved, let us briefly explain how the approach in \cite{MengeshaScott} can be slightly modified in order to obtain the local result given by Proposition \ref{prop:Gehring}. \par 
	First of all, while in \cite{MengeshaScott} the equation is required to hold on the whole space $\mathbb{R}^n$, in the aforementioned paper the equation is in fact only used to prove the Caccioppoli-type inequality given by \cite[Theorem 3.1]{MengeshaScott}, where the equation is tested with test functions which are supported in the ball where the estimate is proved. Therefore, it is clearly sufficient to assume that the equation holds locally. \par 
	On a similar note, while in \cite{MengeshaScott} the estimate \eqref{eq:HDE} is only proved with a slightly larger tail-type quantity appearing on the right-hand side, an inspection of the proof shows that the estimate remains valid if we replace this tail-type quantity by the standard tail appearing in \eqref{eq:HDE}, which is more suitable in the context of obtaining local estimates as we do in the present paper. In fact, the different type of tails in \cite{MengeshaScott} arise by splitting the standard tail appearing in the Caccioppoli inequality \cite[Theorem 3.1]{MengeshaScott} into annuli and using H\"older's inequality to raise the exponent $p-1$ from the standard tail to $p$. However, it is not difficult to see that at any stage of the proof this splitting into annuli is only necessary until the inflated balls exit the ball in which the estimate is proved, since from that point on no additional decay coming from the lack of singularity of the kernel in the far-away regime is necessary to prove the desired estimate. 
	More concretely, for the balls $\mathcal{B}_j:=\mathcal{B}(x_j,R(x_j))$ as defined in \cite[Formula (4.11)]{MengeshaScott1}, in view of the reasoning explained above it is straightforward to deduce reverse H\"older inequalities in the spirit of \cite[Proposition 4.3]{MengeshaScott}, which are of the form
	\begin{equation} \label{eq:MSmod}
			\left (\dashint_{\mathcal{B}_j} U^p d\mu \right )^{1/p} \leq C \Bigg [ \frac{1}{\sigma} \left ( \dashint_{2 \mathcal{B}_j} U^\eta d\mu \right )^{1/\eta} + \sigma Tail(x_j,R(x_j)) \Bigg ] ,
	\end{equation}
	where the definition
	\begin{equation} \label{eq:modTail}
	\begin{aligned}
		\textnormal{Tail}(x_j,R(x_j))&:=\sum_{k=1}^{m_j} 2^{-k(\frac{s}{p-1}-\varepsilon)} \left (\dashint_{2^k \mathcal{B}_{j}} U^\eta d\mu \right )^{1/\eta} \\ & \quad + R(x_j)^{-s-\varepsilon} 2^{-\frac{sp m_j}{p-1}} \textnormal{Tail}(u-(u)_{2^{m_j} B_j};x_j,2^{m_j}R(x_j))
	\end{aligned}
	\end{equation}
	replaces the corresponding one in \cite[Formula (4.2)]{MengeshaScott} and $m_j := \max \{k \in \mathbb{N} \mid 2^k \mathcal{B}_j \subset B_r(x_0) \times B_r(x_0) \}$. Moreover, all notations appearing in \eqref{eq:MSmod} and \eqref{eq:modTail} that we did not define are defined as in \cite{MengeshaScott,MengeshaScott1}. In fact, \eqref{eq:MSmod} can be deduced by simply following the proof of \cite[Proposition 4.3]{MengeshaScott} and splitting the tail in the term $I_2$ as done in \cite[Formula (4.9)]{MengeshaScott} into annuli only up to index $m_j$. \par
	Armed with the modified reverse H\"older inequality \eqref{eq:MSmod} and the modified definition \eqref{eq:modTail}, it is then straightforward to adapt the analysis on the diagonal from \cite[Section 4.2]{MengeshaScott1} and in particular \cite[Formula (4.19)]{MengeshaScott1} in a way that yields the localized estimate \eqref{eq:HDE} as desired.
	
	\section{Comparison estimates}
	
	The following lemma is going to enable us to exploit our small BMO assumption on $A$ in order to freeze the coefficient.
	\begin{lemma} \label{lem:freeze}
		Let $s \in (0,1)$, $p \in [2,\infty)$, $x_0 \in \mathbb{R}^n$, $r >0$ and $A \in \mathcal{L}_0(\Lambda)$. In addition, let $u \in W^{s,p}(B_{4r}(x_0)) \cap L^{p-1}_{sp}(\mathbb{R}^n)$ be a weak solution of the equation
		\begin{equation} \label{eq:1}
			(-\Delta )^s_{p,A} u = 0 \text{ in } B_{4r}(x_0),
		\end{equation}
		and let $v \in W^{s,p}(B_{4r}(x_0)) \cap L^{p-1}_{sp}(\mathbb{R}^n)$ be the unique weak solution of
		\begin{equation} \label{eq:2}
			\begin{cases} \normalfont
				(-\Delta )^s_{p,\widetilde A} v = 0 & \text{ in } B_{r}(x_0) \\
				v = u & \text{ a.e. in } \mathbb{R}^n \setminus B_{r}(x_0),
			\end{cases}
		\end{equation}
		where $\widetilde A \in \mathcal{L}_0(\Lambda)$ is a coefficient satisfying $\widetilde A=A$ in $(\mathbb{R}^n \times \mathbb{R}^n) \setminus (B_{r}(x_0) \times B_{r}(x_0))$.
		Then the function $w:=u-v \in W^{s,p}_0(B_{r}(x_0))$
		satisfies 
		\begin{align*}
			& \left (r^{-n} \int_{\mathbb{R}^n} |w|^p dx \right )^\frac{1}{p} + r^s \left (r^{-n} \int_{\mathbb{R}^n} \int_{\mathbb{R}^n} \frac{|w(x)-w(y)|^p}{|x-y|^{n+sp}}dydx \right )^\frac{1}{p} \\
			& \leq C \widetilde \omega(A-\widetilde A;x_0,r)^\gamma E(u;x_0,4r),
		\end{align*}
		where $C=C(n,s,p,\Lambda)>0$, $\gamma=\gamma(n,s,p,\Lambda) \in (0,1)$ and $$ \widetilde \omega(A-\widetilde A;x_0,r) := \dashint_{B_r(x_0)} \dashint_{B_r(x_0)} |A(x,y)- \widetilde A(x,y)|dydx .$$
	\end{lemma}
	
	\begin{proof} 
		First of all, note that the function $v$ that uniquely solves (\ref{eq:2}) exists by \cite[Remark 3]{existence}. By \cite[Formula (A.4)]{BLS}, for all $X,Y \in \mathbb{R}$ we have
		\begin{equation} \label{eq:elementary}
		|X-Y|^p \leq c\,(|X|^{p-2}X-|Y|^{p-2}Y)(X-Y),
		\end{equation}
		where $c=c(p)>0$.
		Using $w$ as a test function in (\ref{eq:2}) and (\ref{eq:1}) (which is possible in view of a standard density argument) and taking into account that $A(x,y)=\widetilde A(x,y)$ whenever $(x,y) \notin B_{r}(x_0) \times B_{r}(x_0)$ as well as \eqref{eq:elementary}, we obtain
		\begin{align*}
			& \int_{\mathbb{R}^n} \int_{\mathbb{R}^n} \frac{(w(x)-w(y))^p}{|x-y|^{n+sp}}dydx \\
			& \leq \Lambda \int_{\mathbb{R}^n} \int_{\mathbb{R}^n} \widetilde A(x,y) \frac{|(u(x)-u(y))-(v(x)-v(y))|^p}{|x-y|^{n+sp}}dydx \\
			& \leq C \bigg ( \int_{\mathbb{R}^n} \int_{\mathbb{R}^n} \widetilde A(x,y) \frac{|u(x)-u(y)|^{p-2}(u(x)-u(y))(w(x)-w(y))}{|x-y|^{n+sp}}dydx \\
			& \quad - \underbrace{\int_{\mathbb{R}^n} \int_{\mathbb{R}^n} \widetilde A(x,y) \frac{ |v(x)-v(y)|^{p-2}(v(x)-v(y))(w(x)-w(y))}{|x-y|^{n+sp}}dydx}_{=0} \bigg ) \\
			& = C \bigg (\int_{\mathbb{R}^n} \int_{\mathbb{R}^n} (\widetilde A(x,y)-A(x,y)) \frac{|u(x)-u(y)|^{p-2}(u(x)-u(y))(w(x)-w(y))}{|x-y|^{n+sp}}dydx \\
			& \quad + \underbrace{\int_{\mathbb{R}^n} \int_{\mathbb{R}^n} A(x,y) \frac{|u(x)-u(y)|^{p-2}(u(x)-u(y))(w(x)-w(y))}{|x-y|^{n+sp}}dydx}_{=0} \bigg ) \\
			& = C \underbrace{\int_{B_{r}(x_0)} \int_{B_{r}(x_0)} (\widetilde A(x,y)-A(x,y)) \frac{|u(x)-u(y)|^{p-2}(u(x)-u(y))(w(x)-w(y))}{|x-y|^{n+sp}}dydx}_{=: I} .
		\end{align*}
		Let $\varepsilon>0$ to be chosen small enough.
		By using H\"older's inequality twice, first with conjugated exponents $\frac{p}{p-1}$ and $p$ and then another time with conjugated exponents $\frac{p+\varepsilon}{\varepsilon}$ and $\frac{p+\varepsilon}{p}$, we estimate $I$ as follows
		\begin{align*}
			I & \leq C \left (\int_{B_{r}(x_0)} \int_{B_{r}(x_0)} |\widetilde A(x,y)-A(x,y)|^\frac{p}{p-1} \frac{|u(x)-u(y)|^{p}}{|x-y|^{n+sp}}dydx \right )^\frac{p-1}{p} \\
			& \times \left ( \int_{B_{r}(x_0)} \int_{B_{r}(x_0)} \frac{|w(x)-w(y)|^p}{|x-y|^{n+sp}}dydx \right )^\frac{1}{p} \\
			& \leq C \left (\int_{B_{r}(x_0)} \int_{B_{r}(x_0)}  |\widetilde A(x,y)-A(x,y)|^{\frac{p^2}{\varepsilon(p-1)}+\frac{p}{p-1}} dydx \right )^\frac{(p-1)\varepsilon}{p(p+\varepsilon)} \\ & \quad \times r^\frac{n(p-1)}{p+\varepsilon} \left ( \dashint_{B_{r}(x_0)} \int_{B_{r}(x_0)} \frac{|u(x)-u(y)|^{p+\varepsilon}}{|x-y|^{(n+sp)(p+\varepsilon)/p}}dydx \right )^\frac{p-1}{p+\varepsilon} \\
			&  \quad \times\left ( \int_{\mathbb{R}^n} \int_{\mathbb{R}^n} \frac{|w(x)-w(y)|^p}{|x-y|^{n+sp}}dydx \right )^\frac{1}{p} .
		\end{align*}
		Let $\tau=\tau(n,s,p,\Lambda)>0$ be given by Proposition \ref{prop:Gehring}. Then choosing $ \varepsilon:= \tau/n,$ we have in particular $(n+sp)(p+\varepsilon)/p = n+\left (s+\frac{\tau}{p(p+\tau/n)} \right)(p+\tau/n)$ and therefore together with \cite[Proposition 2.5]{MeI} and Proposition \ref{prop:Gehring} with $c=(u)_{B_{4r}(x_0)}$,
		\begin{align*}
			& r^\frac{n(p-1)}{p+\varepsilon} \left ( \dashint_{B_{r}(x_0)} \int_{B_{r}(x_0)} \frac{|u(x)-u(y)|^{p+\varepsilon}}{|x-y|^{(n+sp)(p+\varepsilon)/p}}dydx \right )^\frac{p-1}{p+\varepsilon} \\
			& = C r^\frac{n(p-1)}{p+\tau/n} \left ( \dashint_{B_{r}(x_0)} \int_{B_{r}(x_0)} \frac{|u(x)-u(y)|^{p+\tau/n}}{|x-y|^{n+\left (s+\frac{\tau}{p(p+\tau/n)} \right)(p+\tau/n)}}dydx \right )^\frac{p-1}{p+\tau/n} \\
			& \leq C r^{\frac{n(p-1)}{p+\tau/n}+\tau(p-1)-\frac{\tau(p-1)}{p(p+\tau/n)}} \left ( \dashint_{B_{r}(x_0)} \int_{B_{r}(x_0)} \frac{|u(x)-u(y)|^{p+\tau}}{|x-y|^{n+(s+\tau)(p+\tau)}}dydx \right )^\frac{p-1}{p+\tau} \\
			& \leq C r^{\frac{n(p-1)}{p+\tau/n}-\frac{\tau(p-1)}{p(p+\tau/n)}} \\
			& \quad \times \Bigg [ \left ( \dashint_{B_{2r}(x_0)} \int_{B_{2r}(x_0)} \frac{|u(x)-u(y)|^{p}}{|x-y|^{n+sp}}dydx \right )^\frac{p-1}{p} + \textnormal{Tail}(u-(u)_{B_{4r}(x_0)};x_0,r) \Bigg ] .
		\end{align*}
		Combining the previous two displays with an application of the upper bounds for $A$ and $\widetilde A$ yields
		\begin{align*}
			I & \leq C r^{n(p-1)/p} \left (\dashint_{B_{r}(x_0)} \dashint_{B_{r}(x_0)} |\widetilde A(x,y)-A(x,y)| dydx \right )^\frac{(p-1)\tau/n}{p(p+\tau/n)} \\ & \quad \times \Bigg [ \left ( \dashint_{B_{2r}(x_0)} \int_{B_{2r}(x_0)} \frac{|u(x)-u(y)|^{p}}{|x-y|^{n+sp}}dydx \right )^\frac{p-1}{p} + \textnormal{Tail}(u-(u)_{B_{4r}(x_0)};x_0,r) \Bigg ] \\
			&  \quad \times\left ( \int_{\mathbb{R}^n} \int_{\mathbb{R}^n} \frac{|w(x)-w(y)|^p}{|x-y|^{n+sp}}dydx \right )^\frac{1}{p} .
		\end{align*}
		
		By combining the above estimate for the integral $I$ with the first display of the proof and Proposition \ref{prop:cacc}, we arrive at
		\begin{align*}
			r^s \left (r^{-n} \int_{\mathbb{R}^n} \int_{\mathbb{R}^n} \frac{|w(x)-w(y)|^p}{|x-y|^{n+sp}}dydx \right )^\frac{1}{p} \leq C \widetilde \omega(A-\widetilde A;x_0,r)^\frac{\tau/n}{p+\tau/n} E(u;x_0,4r) .
		\end{align*}
		Together with the Poincar\'e-type inequality \cite[Lemma 4.7]{CozziJFA}, we also obtain
		\begin{align*}
			& \left (r^{-n} \int_{\mathbb{R}^n} |w|^p dx \right )^\frac{1}{p} = C \left (\dashint_{B_{r}(x_0)} |w|^p dx \right )^\frac{1}{p} \\
			& \leq C r^{s} \left (r^{-n} \int_{\mathbb{R}^n} \int_{\mathbb{R}^n} \frac{|w(x)-w(y)|^p}{|x-y|^{n+sp}}dydx \right )^\frac{1}{p} \\
			& \leq C \widetilde \omega(A-\widetilde A;x_0,r)^\frac{\tau/n}{p+\tau/n} E(u;x_0,4r).
		\end{align*}
		Therefore, the desired estimate holds with $\gamma:=\frac{\tau/n}{p+\tau/n} \in (0,1)$.
	\end{proof}
	
	For the remainder of the paper, we fix some bounded domain $\Omega \subset \mathbb{R}^n$. 
	
	We also have the following comparison estimate, which follows from \cite[Lemma 3.3]{KMS2}.
	
	\begin{proposition} \label{prop:meascomp} 
		Let $s \in (0,1)$, $p \in [2,\infty)$, $r>0$, $x_0 \in \mathbb{R}^n$ such that $B_{2r}(x_0) \subset \Omega$ and consider some $A \in \mathcal{L}_0(\Lambda)$.
		Moreover, let $f$ be as in Definition \ref{def:weaksol}. In addition, assume that $u \in W^{s,p}(\Omega) \cap L^{p-1}_{sp}(\mathbb{R}^n)$ is a weak solution of $(-\Delta )^s_{p,A} u = f$ in $\Omega$ and consider the weak solution $u_0 \in W^{s,p}(B_{2r}(x_0)) \cap L^{p-1}_{sp}(\mathbb{R}^n)$ of
		\begin{equation} \label{eq:3}
			\begin{cases} \normalfont
				(-\Delta )^s_{p,A} u_0 = 0 & \text{ in } B_{r}(x_0) \\
				u_0 = u & \text{ a.e. in } \mathbb{R}^n \setminus B_{r}(x_0).
			\end{cases}
		\end{equation}
		Then for any $q \in \Big [1,\frac{n(p-1)}{n-sp} \Big )$, we have the comparison estimate
		$$
		\left (\dashint_{B_{r}(x_0)} |u-u_0|^qdx \right )^{1/q} \leq C \left (r^{sp} \dashint_{B_r(x_0)} |f|dx \right )^\frac{1}{p-1},
		$$
		where $C$ depends only on $n,s,p,\Lambda$ and $q$.
	\end{proposition}
	
	We remark that while in \cite{KMS2} $u$ is assumed to belong to $W^{s,p}(\mathbb{R}^n)$, an inspection of the proof shows that it is sufficient to assume that $u \in W^{s,p}(B_{2r}(x_0)) \cap L^{p-1}_{sp}(\mathbb{R}^n)$ as stated above, since the existence of a unique weak solution $u_0 \in W^{s,p}(B_{2r}(x_0)) \cap L^{p-1}_{sp}(\mathbb{R}^n)$ to \eqref{eq:3} is in that case guaranteed by \cite[Remark 3]{existence}. Moreover, in \cite{KMS2} it is additionally assumed that $f$ is smooth, which is however also only used to obtain the existence of weak solutions in a simpler fashion than done in \cite{existence} and therefore not necessary in order for the proof of \cite[Lemma 3.3]{KMS2} to go through.
	
	Combining the previous two comparison estimates yields the following one.
	
	\begin{corollary} \label{cor:comparisonsum}
		Let $s \in (0,1)$, $p \in [2,\infty)$, $x_0 \in \mathbb{R}^n$, $r >0$ such that $B_{8r}(x_0) \subset \Omega$, and let $f$ be as in Definition \ref{def:weaksol}. Moreover, assume that $A \in \mathcal{L}_0(\Lambda)$. In addition, let $u \in W^{s,p}(\Omega) \cap L^{p-1}_{sp}(\mathbb{R}^n)$ be a weak solution of the equation
		$
			(-\Delta )^s_{p,A} u = f \text{ in } \Omega,
		$
		and let $v \in W^{s,p}(B_{8r}(x_0)) \cap L^{p-1}_{sp}(\mathbb{R}^n)$ be the unique weak solution of
		$$
			\begin{cases} \normalfont
				(-\Delta )^s_{p,\widetilde A_{r,x_0}} v = 0 & \text{ in } B_{r}(x_0) \\
				v = u & \text{ a.e. in } \mathbb{R}^n \setminus B_{r}(x_0),
			\end{cases}
		$$
		where $$\widetilde A_{r,x_0}(x,y):=\begin{cases} \normalfont
			(A)_{r,x_0} & \text{ if } (x,y) \in B_{r}(x_0) \times B_{r}(x_0) \\
			A(x,y) & \text{ if } (x,y) \notin B_{r}(x_0) \times B_{r}(x_0).
		\end{cases}$$
		Then we have the comparison estimate
		\begin{align*}
			& \left (r^{-n} \int_{\mathbb{R}^n} |u-v|^{p-1} dx \right )^\frac{1}{p-1} \\
			& \leq C \omega(A;x_0,r)^\gamma E(u;x_0,4r) + C \left (r^{sp} \dashint_{B_{4r}(x_0)} |f|dx \right )^\frac{1}{p-1},
		\end{align*}
		where $C=C(n,s,p,\Lambda)>0$, $\gamma=\gamma(n,s,p,\Lambda) \in (0,1)$ is given by Lemma \ref{lem:freeze} and
		$$ \omega(A;x_0,r) := \dashint_{B_{r}(x_0)} \dashint_{B_{r}(x_0)} |A(x,y)- (A)_{r,x_0}|dydx .$$
	\end{corollary}
	
	\begin{proof}
		Consider also the weak solution $u_0 \in W^{s,p}(B_{8r}(x_0)) \cap L^{p-1}_{sp}(\mathbb{R}^n)$ of
		$$
			\begin{cases} \normalfont
				(-\Delta )^s_{p,A} u_0 = 0 & \text{ in } B_{4r}(x_0) \\
				u_0 = u & \text{ a.e. in } \mathbb{R}^n \setminus B_{4r}(x_0).
			\end{cases}
		$$
		Since $\widetilde A_{r,x_0}$ satisfies $\widetilde A_{r,x_0}=A$ in $(\mathbb{R}^n \times \mathbb{R}^n) \setminus (B_{r}(x_0) \times B_{r}(x_0))$, in view of Lemma \ref{lem:freeze} and Proposition \ref{prop:meascomp}, we have
		\begin{align*}
			& \left (r^{-n} \int_{\mathbb{R}^n} |u-v|^{p-1} dx \right )^\frac{1}{p-1} \\
			& \leq \left (r^{-n} \int_{\mathbb{R}^n} |u_0-v|^{p-1} dx \right )^\frac{1}{p-1} + \left (r^{-n} \int_{\mathbb{R}^n} |u-u_0|^{p-1} dx \right )^\frac{1}{p-1} \\
			& \leq C \omega(A;x_0,r)^\gamma E(u_0;x_0,4r) + C \left (r^{sp} \dashint_{B_r(x_0)} |f|dx \right )^\frac{1}{p-1} \\
			& \leq C \omega(A;x_0,r)^\gamma E(u;x_0,4r) + C \left (r^{sp} \dashint_{B_{4r}(x_0)} |f|dx \right )^\frac{1}{p-1},
		\end{align*}
		where we also used that $u_0=u$ in $\mathbb{R}^n \setminus B_r(x_0)$.
	\end{proof}
	
	\section{Pointwise estimates for fractional sharp maximal functions} \label{sec:pmf}
	The following excess decay lemma encodes the higher-order regularity properties of solutions to \eqref{eq:splapeq} in a precise way and will be a central tool in order to derive the fine regularity results stated in Section \ref{sec:mr}.
	
	\begin{lemma}[Excess decay] \label{lem:OscDec} 
		Let $s \in (0,1)$, $p \in [2,\infty)$, $R >0$, $x_0 \in \mathbb{R}^n$ such that $B_R(x_0) \subset \Omega$. Moreover, let $f$ be as in Definition \ref{def:weaksol} and suppose that $A \in \mathcal{L}_0(\Lambda)$. In addition, fix some $t_0 \in \big (0,\min \big \{\frac{sp}{p-1},1\big \} \big)$. Then for any weak solution $u \in W^{s,p}(\Omega) \cap L^{p-1}_{sp}(\mathbb{R}^n)$ of $(-\Delta )^s_{p,A} u = f$ in $\Omega$ and any $\rho \in (0,1]$, we have
			\begin{equation} \label{eq:gradoscdecay2}
			\begin{aligned}
				E(u;x_0,\rho R) & \leq C (\rho^{t_0}+\rho^{-\frac{n}{p-1}} \omega(A;x_0,R)^\gamma) E(u;x_0,R) \\
				& \quad + C \rho^{-\frac{n}{p-1}} \left ( R^{sp} \dashint_{B_{R}(x_0)} |f|dx \right )^\frac{1}{p-1},
			\end{aligned}
			\end{equation}
		where $C$ depends only on $n,s,p,\Lambda,t_0$, $\gamma=\gamma(n,s,p,\Lambda) \in (0,1)$ is given by Lemma \ref{lem:freeze} and
		$ \omega(A;x_0,R)$ is defined as in Corollary \ref{cor:comparisonsum}.
	\end{lemma}
	
	\begin{proof}
		If $\rho \geq 2^{-6}$, then by an elementary computation similar to \cite[Lemma 2.4]{KMS2}, we have
		\begin{equation} \label{eq:ex}
			E(u;x_0,\rho R) \leq C \rho^{t_0} E(u;x_0,R),
		\end{equation}
		where $C$ does not depend on $\rho$.

		Next, assume that $\rho \in (0,2^{-6})$ and let $$N_\rho:= \max \{N \in \mathbb{N} \mid 2^{N} \rho < 2^{-5} \}.$$ 
			Consider the weak solution $v \in W^{s,p}(B_{R}(x_0)) \cap L^{p-1}_{sp}(\mathbb{R}^n)$ of
		$$
			\begin{cases} \normalfont
				(-\Delta )^s_{p,\widetilde A_{R/8,x_0}} v = 0 & \text{ in } B_{R/8}(x_0) \\
				v = u & \text{ a.e. in } \mathbb{R}^n \setminus B_{R/8}(x_0),
			\end{cases}
		$$
		where $\widetilde A_{R/8,x_0}$ is defined as in Corollary \ref{cor:comparisonsum}. 
		
		In view of Corollary \ref{cor:comparisonsum} and Remark \ref{rem:infc}, we have
		\begin{align*}
			E(u-v;x_0,\rho R) & \leq C \rho^{-\frac{n}{p-1}} \left (R^{-n} \int_{\mathbb{R}^n} |u-v|^{p-1} dx \right )^\frac{1}{p-1} \\
			& \leq C \rho^{-\frac{n}{p-1}} \omega(A;x_0,R)^\gamma E(u;x_0,R) \\ & \quad + C \rho^{-\frac{n}{p-1}} \left (R^{sp} \dashint_{B_{R}(x_0)} |f|dx \right )^\frac{1}{p-1}.
		\end{align*}
	
		On the other hand, splitting into annuli yields
		\begin{align*}
			E(v;x_0,\rho R) & = \left ( \dashint_{B_{\rho R}(x_0)} |v - (v)_{B_{\rho R}(x_0)}|^{p-1} dx \right )^\frac{1}{p-1} \\
			& \quad + \rho^\frac{sp}{p-1} \sum_{k=1}^{N_\rho} \left (R^{sp} \int_{B_{2^k \rho R}(x_0) \setminus B_{2^{k-1} \rho R}(x_0)} \frac{|v  - (v)_{B_{\rho R}(x_0)}|^{p-1}}{|x_0-y|^{n+sp}}dy \right)^\frac{1}{p-1} \\
			& \quad + \rho^\frac{sp}{p-1} \left ( R^{sp} \int_{\mathbb{R}^n \setminus B_{2^{N_\rho} \rho R}(x_0)} \frac{|v - (v)_{B_{\rho R}(x_0)}|^{p-1}}{|x_0-y|^{n+sp}} dx \right )^\frac{1}{p-1} \\
			& \leq C \underbrace{\sum_{k=0}^{N_\rho} 2^{-\frac{sp}{p-1}k} \left (\dashint_{B_{2^k \rho R}(x_0)} |v  - ( v )_{B_{\rho R}(x_0)}|^{p-1}dy \right)^\frac{1}{p-1}}_{=:I_1} \\
			& \quad + \underbrace{\rho^\frac{sp}{p-1} \left ( R^{sp} \int_{\mathbb{R}^n \setminus B_{2^{N_\rho} \rho R}(x_0)} \frac{|v - (v)_{B_{\rho R}(x_0)}|^{p-1}}{|x_0-y|^{n+sp}} dx \right )^\frac{1}{p-1}}_{=:I_2}.
		\end{align*}
		
		First of all, using Proposition \ref{prop:lb} together with the assumption that $t_0<\frac{sp}{p-1}$, Corollary \ref{cor:comparisonsum} and \eqref{eq:ex}, we obtain
		\begin{align*}
		I_2 & \leq C \rho^\frac{sp}{p-1} \left ( \textnormal{Tail} (v-(v)_{B_{2^{N_\rho} \rho R}(x_0)};x_0,2^{N_\rho} \rho R) + |(v)_{B_{\rho R}(x_0)}-(v)_{B_{2^{N_\rho} \rho R}(x_0)} | \right ) \\
		& \leq C \rho^\frac{sp}{p-1} \left ( \textnormal{Tail} (v-(v)_{B_{2^{N_\rho} \rho R}(x_0)};x_0,2^{N_\rho} \rho R) + \norm{v - (v)_{B_{2^{N_\rho} \rho R}(x_0)}}_{L^\infty(B_{\rho R}(x_0))} \right ) \\
		& \leq C \rho^\frac{sp}{p-1} E(v;x_0,2^{N_\rho} \rho R) \\
		& \leq C \left ( \rho^{t_0} + \rho^{-\frac{n}{p-1}} \omega(A,R,x_0)^\gamma \right ) E(u;x_0,R) + C \rho^{-\frac{n}{p-1}} \left (R^{sp} \dashint_{B_{R}(x_0)} |f|dx \right )^\frac{1}{p-1}.
		\end{align*}
		
		Since $\widetilde A_{R/8,x_0}$ is constant in $B_{R/8}(x_0)$, we can apply Proposition \ref{prop:Holdreg} with $\alpha=t_0$, which yields that for any $k \in \{0,...,N_\rho\}$, we have
		\begin{align*}
			\left ( \dashint_{B_{2^k \rho R}(x_0)} |v - \left (v \right )_{B_{2^k \rho R}(x_0)}|^{p-1} dx \right )^\frac{1}{p-1} & \leq \osc_{B_{2^k \rho R}(x_0)} v \\
			& \leq C 2^{k t_0} \rho^{t_0} E(v;x_0,R/8).
		\end{align*}
	Therefore, together with Corollary \ref{cor:comparisonsum} and \eqref{eq:ex} and taking into account that $t_0<\frac{sp}{p-1}$, for $I_1$ we deduce
		\begin{align*}
		I_1 & \leq C \rho^{t_0} E(v;x_0,R/8) \underbrace{\sum_{k=0}^\infty 2^{k \left (t_0-\frac{sp}{p-1} \right )}}_{<\infty} \\
		& \leq C \left ( \rho^{t_0} + \rho^{-\frac{n}{p-1}} \omega(A,R,x_0)^\gamma \right ) E(u;x_0,R) + C \rho^{-\frac{n}{p-1}} \left (R^{sp} \dashint_{B_{R}(x_0)} |f|dx \right )^\frac{1}{p-1}.
	\end{align*}
	
	Combining all the above estimates now leads to
	\begin{align*}
		& E(u;x_0,\rho R) \\ & \leq E(v;x_0,\rho R) + E(u-v;x_0,\rho R) \\
		& \leq C \left ( \rho^{t_0} + \rho^{-\frac{n}{p-1}} \omega(A,R,x_0)^\gamma \right ) E(u;x_0,R) + C \rho^{-\frac{n}{p-1}} \left (R^{sp} \dashint_{B_{R}(x_0)} |f|dx \right )^\frac{1}{p-1},
	\end{align*}
	finishing the proof.
	\end{proof}

Lemma \ref{lem:OscDec} in particular enables us to derive pointwise estimates in terms of fractional maximal functions as defined in Section \ref{Cald}. However, due to the nonlocal nature of our setting, it is more natural to derive such estimates first in terms of \emph{nonlocal fractional sharp maximal functions} which are defined in terms of the nonlocal excess functional as follows.
\begin{definition}[Nonlocal fractional sharp maximal function]
	Let $s \in (0,1)$ and $p \in [2,\infty)$. Given $x_0 \in \mathbb{R}^n$, $R>0$ and $u \in L^{p-1}_{sp}(\mathbb{R}^n)$, we define the nonlocal fractional sharp maximal function ${N}^\#_{R,t}$ of order $t \in [0,1]$ by
	\begin{equation} \label{eq:nonlocmax}
		{N}^\#_{R,t} (u)(x_0):= \sup_{0<r\leq R} r^{-t} E(u;x_0,r).
	\end{equation}
\end{definition}
Note that by definition and H\"older's inequality, we clearly have
\begin{equation} \label{eq:locestnl}
	{M}^\#_{R,t} (u)(x_0) \leq {N}^\#_{R,t} (u)(x_0).
\end{equation}

\begin{remark}[Convolution form] \normalfont
	We remark that for the convolution kernel
	$$\psi:\mathbb{R}^n \to \mathbb{R}, \quad \psi(x):=\frac{1}{\left (1+|x| \right)^{n+sp}}$$
	and the associated family $\{\psi_r\}_{r>0}$ defined by $\psi_r(x):=r^{-n} \psi(x/r)$, 
	it is straightforward to check that for any $r>0$ and any $x_0 \in \mathbb{R}^n$ we have
	$$ E(u;x_0,r) \asymp (|u-(u)_{B_r(x_0)}|^{p-1} \ast \psi_r)^\frac{1}{p-1}(x_0)$$
	and therefore
	\begin{equation} \label{eq:convform}
	{N}^\#_{R,t} (u)(x_0) \asymp \sup_{0<r\leq R} r^{-t} (|u-(u)_{B_r(x_0)}|^{p-1} \ast \psi_r)^\frac{1}{p-1}(x_0). 
	\end{equation}
	While from the point of view of our proof the definition \eqref{eq:nonlocmax} of ${N}^\#_{R,t}$ is the natural one, the convolution form on the right-hand side of \eqref{eq:convform} makes a clearer connection to maximal-type functions arising in classical harmonic analysis, see e.g.\ \cite{FefStein}.
\end{remark}

We now turn to proving the mentioned pointwise maximal function estimates.
\begin{theorem}[Pointwise maximal function bounds] \label{thm:fracmaxest}
	Let $s \in (0,1)$, $p \in [2,\infty)$, $R >0$, $x_0 \in \mathbb{R}^n$ such that $B_{R}(x_0) \subset \Omega$ and let $f$ be as in Definition \ref{def:weaksol}. Moreover, fix some $t \in \big (0,\min \big \{\frac{sp}{p-1},1 \big \} \big)$. Then there exists some small enough $\delta=\delta(n,s,p,\Lambda,t) \in (0,1)$ such that if $A \in \mathcal{L}_0(\Lambda)$ is $\delta$-vanishing in $B_{R}(x_0)$, then for any weak solution $u \in W^{s,p}(\Omega) \cap L^{p-1}_{sp}(\mathbb{R}^n)$ of $(-\Delta )^s_{p,A} u = f$ in $\Omega$, we have
\begin{equation} \label{eq:fmd}
		{N}^\#_{R,t}  (u)(x_0) \leq C \left ( R^{-t} E(u;x_0,R) + \left ({M}_{R,{sp-t(p-1)}} (f)(x_0) \right )^\frac{1}{p-1} \right ),
\end{equation}
where $C$ depends only on $n,s,p,\Lambda,t$.
\end{theorem}

\begin{proof}
	Fix some $\varepsilon \in (0,R)$ and consider the following modified nonlocal fractional sharp maximal function
	$$
		{N}^{\#,\varepsilon}_{R,t} (u)(x_0):= \sup_{\varepsilon \leq r\leq R} \, r^{-t} E(u;x_0,r).
	$$
	Let $\rho \in (0,1)$ to be chosen small enough in a way such that $\rho$ depends only on $n,s,p,\Lambda,t$. Fix some $r \in [\varepsilon,R]$.
	
	If $r/\rho \geq R$, then by an elementary computation, we have
	\begin{align*}
		r^{-t} E(u;x_0,r) \leq C \rho^{-\frac{n}{p-1}-t} R^{-t} E(u;x_0,R) \leq C R^{-t} E(u;x_0,R),
	\end{align*}
	where $C$ depends only on $n,s,p,\Lambda,t$, since $\rho$ depends only on the aforementioned quantities.

	Next, consider the case when $r/\rho<R$.
	Then by Lemma \ref{lem:OscDec} with $R$ replaced by $r/\rho$ and with respect to $$t_0:= \frac{t+\min \left \{\frac{sp}{p-1},1 \right \}}{2} \in \bigg (t,\min \left \{\frac{sp}{p-1},1\right \} \bigg ),$$ we have
	\begin{align*}
		r^{-t} E(u;x_0,r) & \leq C (\rho^{t_0-t}+\rho^{-\frac{n}{p-1}-t} \omega(A;x_0,r/\rho)^\gamma) \left (\frac{r}{\rho} \right )^{-t} E(u;x_0,r/\rho) \\
		& \quad + C \rho^{-\frac{n}{p-1}-t} \left ( \left (\frac{r}{\rho} \right )^{sp-t(p-1)}  \dashint_{B_{r/\rho}(x_0)} |f|dx \right )^\frac{1}{p-1} \\
		& \leq C (\rho^{t_0-t}+\rho^{-\frac{n}{p-1}-t} \delta^\gamma) {N}^{\#,\varepsilon}_{R,t} (u)(x_0) \\
		& \quad + C \rho^{-\frac{n}{p-1}-t} \left ({M}_{R,{sp-t(p-1)}}  (f)(x_0) \right )^\frac{1}{p-1},
	\end{align*} 
	where so far $C$ does not depend on $\rho$ and $\varepsilon$.
	
	Combining the previous display with the second one in the proof, we obtain
	\begin{equation} \label{eq:fmd1}
		\begin{aligned}
			{N}^{\#,\varepsilon}_{R,t} (u)(x_0) & \leq C_0 (\rho^{t_0-t}+\rho^{-\frac{n}{p-1}-t} \delta^\gamma) {N}^{\#,\varepsilon}_{R,t} (u)(x_0) \\
			& \quad + C R^{-t} E(u;x_0,R) \\
			& \quad + C \rho^{-\frac{n}{p-1}-t} \left ({M}_{R,{sp-t(p-1)}}  (f)(x_0) \right )^\frac{1}{p-1},
		\end{aligned}
	\end{equation}
	where $C_0 \geq 1$ does not depend on $\rho$, while $C$ now depends on $m$, but both $C_0$ and $C$ do not depend on $\varepsilon$.
	
	Now choose $\rho$ small enough such that
	$$C_0 \rho^{t_0-t} \leq \frac{1}{4}$$
	and then $\delta \in (0,1)$ small enough such that
	$$ C_0 \rho^{-\frac{n}{p-1}-t} \delta^\gamma \leq \frac{1}{4}.$$
	Since clearly ${N}^{\#,\varepsilon}_{R,t}  (u)(x_0)<\infty$, the above choices of $\rho$ and $\delta$ enable us to reabsorb the first term of the right-hand side of \eqref{eq:fmd1}, which implies the estimate
	\begin{align*}
		{N}^{\#,\varepsilon}_{R,t}  (u)(x_0) \leq C \left ( R^{-t} E(u;x_0,R) + \left ({M}_{R,{sp-t(p-1)}} (f)(x_0) \right )^\frac{1}{p-1} \right ),
	\end{align*}
	where $C$ still does not depend on $\varepsilon$. Therefore, letting $\varepsilon \to 0$ yields the desired estimate \eqref{eq:fmd}.
\end{proof}

\section{Fine regularity estimates}

Combining Theorem \ref{thm:fracmaxest} with the relations of fractional sharp maximal functions with fractional Sobolev spaces from Proposition \ref{prop:embedding} now yields fine regularity estimates in such spaces. The required assumptions on the right-hand side $f$ simply arise from the mapping properties of the standard fractional maximal function given by Proposition \ref{prop:fracmaxest}.

\begin{corollary} \label{cor:CZest1}
	Let $s \in (0,1)$, $p \in [2,\infty)$, $\Lambda \geq 1$ and let $f$ be as in Definition \ref{def:weaksol} with $\Omega$ replaced by $B_1$. In addition, fix some $s \leq t < \min \big \{\frac{sp}{p-1},1 \big \}$ and some $q \in \big (1,\frac{n}{sp-t(p-1)} \big)$. There exists some $\delta=\delta(n,s,p,\Lambda,t)>0$, such that if $A \in \mathcal{L}_0(\Lambda)$ is $\delta$-vanishing in $B_1$, 
	then for any weak solution $u \in W^{s,p}(B_1) \cap L^{p-1}_{sp}(\mathbb{R}^n)$
	of the equation
	$
	(-\Delta )^s_{p,A} u = f \text{ in } B_1
	$
	and $q^\star:=\frac{nq(p-1)}{n-(sp-t(p-1))q}$,
	we have
	\begin{equation} \label{eq:CZest1}
		\begin{aligned}
			\norm{u}_{W^{t,q^\star}(B_{1/16})} & \leq C \left (E(u;0,1) + \norm{f}_{L^q(B_1)}^\frac{1}{p-1} \right ),
		\end{aligned}
	\end{equation}
	where $C$ depends only on $n,s,p,\Lambda,t,q$.
\end{corollary}

\begin{proof}
Fix some $t \in \big [s,\big \{\frac{sp}{p-1},1 \big \} \big )$ and some $q \in \big (1,\frac{n}{sp-t(p-1)} \big)$ and let $\delta=\delta(n,s,p,\Lambda,t)>0$ be given by Theorem \ref{thm:fracmaxest}. Extend $f$ to $\mathbb{R}^n$ by setting $f \equiv 0$ in $\mathbb{R}^n \setminus B_1$.

Fix some small enough $\varepsilon>0$ such that $t_\varepsilon := t+\varepsilon$ belongs to the range (\ref{eq:trange}) and set $$q^\star_\varepsilon:=\frac{nq^\star}{n+\varepsilon q^\star}= \frac{nq(p-1)}{n-(sp-t_\varepsilon(p-1))q}<q^\star.$$

Applying the $L^{q^\star_\varepsilon}$ norm over $B_{1/2}$ to both sides of the estimate \eqref{eq:fmd} from Theorem \ref{thm:fracmaxest} and taking into account \eqref{eq:locestnl}, we obtain
	\begin{align*}
		\norm{{M}^\#_{1/2,\widetilde t}(u)}_{L^{q^\star_\varepsilon}(B_{1/2})} & \leq C \left ( E(u;0,1) + \norm{{M}_{1/2,{sp-t_\varepsilon(p-1)}} (f)}_{L^\frac{q^\star_\varepsilon}{p-1}(\mathbb{R}^n)}^\frac{1}{p-1} \right ).
	\end{align*}

Combined with Proposition \ref{prop:embedding} and the mapping properties of the standard fractional maximal function from Proposition \ref{prop:fracmaxest}, we obtain
\begin{align*}
||u||_{W^{t,q^\star}(B_{1/16})} & \leq C \left (\norm{u}_{L^{q^\star_\varepsilon}(B_{1/2})} + \norm{{M}^\#_{1/2,t_\varepsilon}(u)}_{L^{q^\star_\varepsilon}(B_{1/2})} \right ) \\
& \leq C \left ( \norm{u}_{L^{q^\star}(B_{1/2})} + E(u;0,1) + \norm{{M}_{{sp-t_\varepsilon(p-1)}} (f)}_{L^\frac{q^\star_\varepsilon}{p-1}(\mathbb{R}^n)}^\frac{1}{p-1} \right ) \\
& \leq C \left (\norm{u}_{L^{q^\star}(B_{1/2})} + E(u;0,1) + \norm{f}_{L^q(B_1)}^\frac{1}{p-1} \right ) ,
\end{align*}
where we also used that $$\frac{q^\star_\varepsilon}{p-1} = \frac{nq}{n-(sp-t_\varepsilon(p-1))q}.$$
Since as a consequence of the zero-order potential estimates \cite[Theorem 1.2]{KMS2} (see \cite[Corollary 1.1]{KMS2}) we in particular have the estimate 
$$ \norm{u}_{L^{q^\star}(B_{1/2})} \leq C \left ( E(u;0,1) + \norm{f}_{L^q(B_1)}^\frac{1}{p-1} \right ),$$
the proof is complete.
\end{proof}

Theorem \ref{thm:CZestBMO} now follows from Corollary \ref{cor:CZest1} by scaling and covering arguments. 

\begin{proof}[Proof of Theorem \ref{thm:CZestBMO}]
		Fix a relatively compact domain ${\Omega^\prime} \Subset \Omega$ and let $\delta=\delta(n,s,p,\Lambda,t)>0$ be given by Corollary \ref{cor:CZest1}. There exists some $r \in (0,R_0)$ such that for any $z \in \Omega^\prime$, we have $B_{r}(z) \Subset \Omega$. Since $A$ is $(\delta,R_0)$-BMO in $\Omega$, for any $z \in \Omega^\prime$ we thus obtain that $A$ is $\delta$-vanishing in $B_{r}(z)$. For any $z \in \Omega^\prime$, we define the scaled functions
		\begin{equation} \label{eq:scaled}
		u_z(x):=u(rx+z), \quad f_z(x):=r^{sp} f(rx+z), \quad A_z(x,y):=A(rx+z,ry+z)
		\end{equation}
		and observe that $u_z \in W^{s,p}(B_1) \cap L_{sp}^{p-1}(\mathbb{R}^n)$ is a weak solution of $(-\Delta)^s_{p,A_z} u_z=f_z$ in $B_1$ and that $A_z \in \mathcal{L}_0(\Lambda)$ is $\delta$-vanishing in $B_1$. Therefore, by Corollary \ref{cor:CZest1} and rescaling for any $z \in \Omega^\prime$ we deduce that
		\begin{align*}
			\norm{u}_{W^{t,q^\star}(B_{r/16}(z))} & = \norm{u_z}_{W^{t,q^\star}(B_{1/16})} \\ & \leq C \left ( E(u_z;0,1) + \norm{f_z}_{L^q(B_1)}^\frac{1}{p-1} \right ) \\
			& \leq C \left ( E(u;r,z) + \norm{f}_{L^q(B_r(z))}^\frac{1}{p-1} \right ),
		\end{align*}
		where $C$ depends only on $n,s,p,\Lambda,t,q,r$.
	Since $\left \{B_{r/16}(z) \right \}_{z \in {\Omega^\prime}}$ is an open covering of $\overline {\Omega^\prime}$ and $\overline {\Omega^\prime}$ is compact, there exists a finite subcover $\left \{B_{r/16}(z_i) \right \}_{i=1}^N$ of $\Omega^\prime$. Now fixing some arbitrary reference point $x_0 \in \Omega^\prime$ and summing the above estimates over $i=1,...,N$ yields
	\begin{equation} \label{eq:sumest}
		\begin{aligned}
			\norm{u}_{W^{t,q^\star}(\Omega^\prime)} & \leq \sum_{i=1}^N \norm{u}_{W^{t,q^\star}(B_{r/16}(z_i))} \\
			& \leq C \sum_{i=1}^N \left ( E(u;r,z_i) + \norm{f}_{L^q(B_r(z_i))}^\frac{1}{p-1} \right )\\
			& \leq C \Bigg [ \left (\int_{\Omega} |u|^{p-1} dx \right )^\frac{1}{p-1} + \left (\int_{\mathbb{R}^n \setminus \Omega} \frac{|u(y)|^{p-1}}{|x_0-y|^{n+sp}}dy \right )^\frac{1}{p-1} + \norm{f}_{L^q(\Omega)}^\frac{1}{p-1} \Bigg ],
		\end{aligned}
	\end{equation}
	where $C$ depends only on $n,s,p,\Lambda,t,q,R_0,\Omega^\prime,\Omega$, proving that $u \in W^{t,q}_{\loc}(\Omega)$.
	
	In order to prove the estimate \eqref{eq:CZregBMO}, fix some $R>0$ and some $x_0 \in \Omega$ such that $B_R(x_0) \subset \Omega$ and consider the scaled functions 
	\begin{equation} \label{eq:scaled2}
	u_R(x):=u(Rx), \quad f_R(x):=R^{sp} f(Rx), \quad A_R(x,y):=A(Rx,Ry).
	\end{equation}
	Note that $u_R \in W^{s,p}(B_1(x_0)) \cap L_{sp}^{p-1}(\mathbb{R}^n)$ is a weak solution of $(-\Delta)^s_{p,A_R} u_R=f_R$ in $B_1(x_0)$, while $A_R \in \mathcal{L}_0(\Lambda)$ is $(\delta,R_0/R)$-BMO and thus $(\delta,R_0/ \max \{\textnormal{diam}(\Omega),1\})$-BMO in $B_1(x_0)$.
	Therefore, applying the estimate \eqref{eq:sumest} to $u_R$ with $\Omega$ replaced by $B_1(x_0)$ and $\Omega^\prime$ replaced by $B_{1/2}(x_0)$ along with changing variables yields
	\begin{align*}
		R^{-n/q^\star} [u]_{W^{t,q^\star}(B_{R/2}(x_0))} & = C R^{-t} \left (\int_{B_{1/2}(x_0)} \int_{B_{1/2}(x_0)} \frac{|u_R(x)-u_R(y)|^{q^\star}}{|x-y|^{n+tq^\star}} dydx \right )^{1/q} \\
		& \leq C R^{-t} \left ( E(u_R;x_0,1) + \norm{f_R}_{L^q(B_1(x_0))}^\frac{1}{p-1} \right ) \\
		& \leq C \left (R^{-t} E(u;x_0,R) + R^{-n/q^\star} \norm{f}_{L^q(B_R(x_0))}^\frac{1}{p-1} \right ) ,
	\end{align*}
	proving also the estimate \eqref{eq:CZregBMO}. Thus, the proof is finished.
\end{proof}

Next, we deduce our H\"older regularity result Theorem \ref{thm:fineHold} from Theorem \ref{thm:fracmaxest}.

\begin{proof}[Proof of Theorem \ref{thm:fineHold}]
	Let us first consider the case when $\Omega=B_1$ and $A$ is $\delta$-vanishing in $B_1$ for $\delta$ given as in Theorem \ref{thm:fracmaxest} applied with $t=\alpha$. 
	
	Extend $f$ to $\mathbb{R}^n$ by setting $f \equiv 0$ in $\mathbb{R}^n \setminus B_1$.
	By Proposition \ref{prop:sfm}, we have
	\begin{equation} \label{eq:CaldHold}
		[u]_{C^{\alpha}(B_{1/8})} \leq C \sup_{x_0 \in B_{1/8}} {M}^\#_{1/2,\alpha}(u)(x_0).
	\end{equation}
	
	Using the pointwise estimate \eqref{eq:fmd} from Theorem \ref{thm:fracmaxest} and taking into account \eqref{eq:CaldHold}, \eqref{eq:locestnl} as well as Proposition \ref{lem:Linfbd} with $\beta=sp-\alpha(p-1) \in (0,n)$, we obtain
	\begin{equation} \label{eq:scale1}
	\begin{aligned}
		[u]_{C^{\alpha}(B_{1/8})} & \leq C  \sup_{x_0 \in B_{1/8}} {M}^\#_{1/2,\alpha}(u)(x_0) \\
		& \leq C \left ( E(u;0,1) + \sup_{x_0 \in \mathbb{R}^n} {M}_{{sp-\alpha(p-1)}} (f)(x_0) \right ) \\
		& \leq C \left ( E(u;0,1) + \norm{f}_{L^{\frac{n}{sp-\alpha(p-1)},\infty}(B_1)}^\frac{1}{p-1} \right ) .
	\end{aligned}
	\end{equation}
	Next, let us consider the general case when $\Omega$ is an arbitrary bounded domain and when $A$ is $(\delta,R_0)$-BMO in $\Omega$ for $\delta$ given as in Theorem \ref{thm:fracmaxest} applied with $t=\alpha$. Fix a relatively compact domain ${\Omega^\prime} \Subset \Omega$. There exists some $r \in (0,R_0)$ such that for any $z \in \Omega^\prime$, we have $B_{r}(z) \Subset \Omega$. In particular, for any $z \in \Omega^\prime$ $A$ is $\delta$-vanishing in $B_{r}(z)$.
	
	For any $z \in \Omega^\prime$, let the scaled functions $u_z$, $f_z$ and $A_z$ be defined as in \eqref{eq:scaled}, so that $u_z \in W^{s,p}(B_1) \cap L_{sp}^{p-1}(\mathbb{R}^n)$ is a weak solution of $(-\Delta)^s_{p,A_z} u_z=f_z$ in $B_1$ and $A_z \in \mathcal{L}_0(\Lambda)$ is $\delta$-vanishing in $B_1$. Therefore, using \eqref{eq:scale1}, we deduce
	\begin{equation} \label{eq:CaScaled}
	\begin{aligned}
		[u]_{C^{\alpha}(B_{r/8}(z))} = [u_z]_{C^{\alpha}(B_{1/8})} & \leq C \left ( E(u_z;0,1) + \norm{f_z}_{L^{\frac{n}{sp-\alpha(p-1)},\infty}(B_1)}^\frac{1}{p-1} \right ) \\
		& \leq C \left ( E(u;r,z) + \norm{f}_{L^{\frac{n}{sp-\alpha(p-1)},\infty}(B_r(z))}^\frac{1}{p-1} \right ),
	\end{aligned}
	\end{equation}
	where $C$ depends only on $n,s,p,\Lambda,\alpha,r$.
	
	Again, there exists a finite subcover $\left \{B_{r/16}(z_i) \right \}_{i=1}^N \subset \left \{B_{r/16}(z) \right \}_{z \in {\Omega^\prime}}$ of $\Omega^\prime$. Now fix points $x,y \in \Omega^\prime$ with $x \neq y$ and also some arbitrary reference point $x_0 \in \Omega^\prime$. Then $x \in B_{r/16}(z_i)$ for some $i \in \{1,...,N\}$. 
	
	If $|x-y| < r/16$, then $y \in B_{r/8}(z_i)$, so that \eqref{eq:CaScaled} yields
	\begin{align*}
		\frac{|u(x)-u(y)|}{|x-y|^\alpha} & \leq C \left ( E(u;r,z_i) + \norm{f}_{L^{\frac{n}{sp-\alpha(p-1)},\infty}(B_r(z_i))}^\frac{1}{p-1} \right ) \\
		& \leq C \Bigg [ \left (\int_{\Omega} |u|^{p-1} dx \right )^\frac{1}{p-1} + \left ( \int_{\mathbb{R}^n \setminus \Omega} \frac{|u(y)|^{p-1}}{|x_0-y|^{n+sp}}dy \right )^\frac{1}{p-1} \\ & \quad + \norm{f}_{L^{\frac{n}{sp-\alpha(p-1)},\infty}(\Omega)}^\frac{1}{p-1} \Bigg ],
	\end{align*}
	where $C$ depends only on $n,s,p,\Lambda,\alpha,R_0,\Omega^\prime,\Omega$.
	
	If on the other hand $|x-y| \geq r/16$, then we estimate
	$$ \frac{|u(x)-u(y)|}{|x-y|^\alpha} \leq 2 (16/r)^\alpha \sup_{\Omega^\prime} |u|.$$
	
	Since as a well-known consequence of the zero-order potential estimates \cite[Theorem 1.2]{KMS2} we in particular have the estimate
	\begin{align*}
	\sup_{\Omega^\prime} |u| & \leq C \Bigg [ \left (\int_{\Omega} |u|^{p-1} dx \right )^\frac{1}{p-1} + \left ( \int_{\mathbb{R}^n \setminus \Omega} \frac{|u(y)|^{p-1}}{|x_0-y|^{n+sp}}dy \right )^\frac{1}{p-1} \\ & \quad + \norm{f}_{L^{\frac{n}{sp-\alpha(p-1)},\infty}(\Omega)}^\frac{1}{p-1} \Bigg ],
	\end{align*}
	combining the previous three displays yields the estimate
	\begin{equation} \label{eq:genHold}
	\begin{aligned}
		\norm{u}_{C^{\alpha}(\Omega^\prime)} & \leq C \Bigg [ \left (\int_{\Omega} |u|^{p-1} dx \right )^\frac{1}{p-1} + \left ( \int_{\mathbb{R}^n \setminus \Omega} \frac{|u(y)|^{p-1}}{|x_0-y|^{n+sp}}dy \right )^\frac{1}{p-1} \\ & \quad + \norm{f}_{L^{\frac{n}{sp-\alpha(p-1)},\infty}(\Omega)}^\frac{1}{p-1} \Bigg ],
	\end{aligned}
	\end{equation}
	where $C$ depends only on $n,s,p,\Lambda,\alpha,R_0,\Omega^\prime,\Omega$, so that $u \in C^\alpha_{loc}(\Omega)$ as desired. 
	
	In order to prove the estimate \eqref{eq:FineHoldest}, fix some $R>0$ and some $x_0 \in \Omega$ such that $B_R(x_0) \subset \Omega$ and consider the scaled functions $u_R$, $f_R$ and $A_R$ as defined in \eqref{eq:scaled2}, so that $u_R$ is a weak solution of $(-\Delta)^s_{p,A_R} u_R=f_R$ in $B_1(x_0)$, while $A_R \in \mathcal{L}_0(\Lambda)$ is $(\delta,R_0/ \max \{\textnormal{diam}(\Omega),1\})$-BMO in $B_1(x_0)$.
	Therefore, applying the estimate \eqref{eq:genHold} to $u_R$ with $\Omega$ replaced by $B_1(x_0)$ and $\Omega^\prime$ replaced by $B_{1/2}(x_0)$ yields
	\begin{align*}
		R^{\alpha}[u]_{C^\alpha(B_{R/2}(x_0))} & = [u_R]_{C^\alpha(B_{1/2}(x_0))} \\
		& \leq C \left ( E(u_R;x_0,1) + \norm{f_R}_{L^{\frac{n}{sp-\alpha(p-1)},\infty}(B_1(x_0))}^\frac{1}{p-1} \right ) \\ & \leq C \left ( E(u;x_0,R) + R^\alpha \norm{f}_{L^{\frac{n}{sp-\alpha(p-1)},\infty}(B_R(x_0))}^\frac{1}{p-1} \right ) ,
	\end{align*}
	proving the estimate \eqref{eq:FineHoldest}.
\end{proof}

Finally, Corollary \ref{cor:HD} follows directly from Theorem \ref{thm:CZestBMO} and well-known embedding results for fractional Sobolev spaces.
\begin{proof}[Proof of Corollary \ref{cor:HD}]
	Fix some $t \in \left (s, \min \left \{\frac{sp}{p-1},1 \right \} \right )$ and additionally some $t^\prime \in \left (t, \min \left \{\frac{sp}{p-1},1 \right \} \right )$. Since $f \in L^\frac{p}{p-1}(\Omega)$, Theorem \ref{thm:CZestBMO} yields $u \in W^{t^\prime,\widetilde p}_{\loc}(\Omega)$ for $\widetilde p:=\frac{np}{n- \big (\frac{sp}{p-1} - t^\prime \big)p}>p$, which by \cite[Proposition 2.5]{MeI} implies $u \in W^{t,p}_{\loc}(\Omega)$. 
\end{proof}
%\textbf{Data Availability Statement}: Data sharing is not applicable to this article as no data sets were generated or analyzed during the current study.

\printbibliography 
	
\end{document}